\documentclass[12pt,reqno]{amsart}

\usepackage[cp1251]{inputenc}
\usepackage[english]{babel}
\usepackage{amsmath,amsxtra,amssymb,eufrak}
\usepackage[all]{xy}
\usepackage{mathrsfs}
\usepackage[active]{srcltx} 

\newtheorem{thm}{Theorem}[section]
\newtheorem{lm}[thm]{Lemma}
\newtheorem{co}[thm]{Corollary}
\newtheorem{pr}[thm]{Proposition}

\theoremstyle{definition}

\newtheorem{df}[thm]{Definition}
\newtheorem{exm}[thm]{Example}

\newtheorem{rem}[thm]{Remark}
\newtheorem{rems}[thm]{Remarks}

\newtheorem{nota}[thm]{Notation}

\numberwithin{equation}{section}

\DeclareMathOperator{\arcsinh}{arcsinh}

\newcommand*{\wh}{\widehat}
\newcommand*{\wt}{\widetilde}

\newcommand{\Rad}{\mathop{\mathrm{Rad}}\nolimits}

\newcommand{\id}{\mathop{\mathrm{id}}\nolimits}

\newcommand{\ptn}{\mathbin{\widehat{\otimes}}}

\newcommand{\ad}{\mathop{\mathrm{ad}}\nolimits}

\newcommand{\SL}{\mathop{\mathrm{SL}}\nolimits}

{\end{compactenum}}
\newcommand*{\lla}{\longleftarrow}

\newcommand*{\cF}{\mathscr F}

\newcommand{\cO}{\mathcal{O}}

\newcommand*{\cR}{\mathcal R}

\newcommand{\CC}{\mathbb{C}}

\newcommand{\R}{\mathbb{R}}
\newcommand{\Z}{\mathbb{Z}}
\newcommand{\N}{\mathbb{N}}
\newcommand{\T}{\mathbb{T}}

\newcommand*{\fA}{\mathfrak{A}}

\newcommand*{\fa}{\mathfrak{a}}

\newcommand*{\ff}{\mathfrak{f}}

\newcommand*{\fg}{\mathfrak{g}}

\newcommand*{\rT}{\mathrm T}

\renewcommand{\le}{\leqslant}
\renewcommand{\ge}{\geqslant}

\let \al         =\alpha
\let \be         =\beta

\let \de         =\delta

\let \te         =\theta        
\let \io         =\iota
\let \ka         =\varkappa
\let \la         =\lambda

\let \De         =\Delta

\let \Om         =\Omega

\let \phi         =\varphi

\title{Finitely $C^\infty$-generated associative and Hopf algebras}

\author{O. Yu. Aristov}
\keywords{Banach algebra of polynomial growth, envelope with respect to a class of Banach algebras, free $C^\infty$-function, universal enveloping algebra, quantum  SL(2), deformation of the two-dimensional Lie algebra, $C^\infty$-tensor algebra, $C^\infty$-symmetric algebra, finitely $C^\infty$-generated algebra, topological Hopf algebra}
\email{aristovoyu@inbox.ru}
\begin{document}

\maketitle

\markright{Finitely $C^\infty$-generated algebras}

\begin{abstract}
We introduce finitely $C^\infty$-generated algebras, which can be treated as `algebras of functions' on non-commutative $C^\infty$-differentiable spaces. Our approach uses the category of  projective limits of real Banach algebras of polynomial growth. We prove the existence of some universal constructions in this and some similar categories. By analogy with holomorphically finitely generated algebras of Pirkovskii, a finitely $C^\infty$-generated algebra is defined as a quotient of a finite-rank algebra of `free $C^\infty$-functions'. The latter notion was introduced by the author in a previous article, where a structure theorem for algebras of `free $C^\infty$-functions' was announced and proved in dimension at most $2$. Here this theorem is proved in full generality.

The central result asserts that the projective tensor product of a finite tuple of finitely $C^\infty$-generated algebras is finitely $C^\infty$-generated. In particular, this makes it natural to consider finitely $C^\infty$-generated topological Hopf algebras. Furthermore, a construction called `envelope' provides a functor from the category of affine real Hopf algebras to the category of finitely $C^\infty$-generated Hopf algebras.
\end{abstract}

\section*{Introduction}

We consider (in general, non-commutative) algebras that can be treated as the algebras of 'functions of class $C^\infty$' on non-commutative spaces and generalize not only algebras of $C^\infty$ functions on manifolds but also $C^\infty$-differentiable algebras. Algebras of free $C^\infty$-functions (aka $C^\infty$-tensor algebras) were introduced by the author as envelopes of free algebras with respect to the class of Banach algebras of polynomial growth \cite{ArNew}. Here we discuss \emph{finitely $C^\infty$-generated algebras}, i.e.,  quotients of algebras of free $C^\infty$-functions of finite rank. It is shown in \cite{ArNew} that every  $\mathsf{PGL}$ algebra (i.e., a projective limit of real Banach algebras of polynomial growth) is a quotient of a $C^\infty$-tensor algebra but the additional imposition of the finiteness condition  makes the class of quotients best behaved in many ways, e.g., it is well compatible with the projective tensor product. For example, it is not clear whether it makes sense to consider topological Hopf algebras with general underlying $\mathsf{PGL}$ algebra as a separate class, but in \S\,\ref{s:Hopf} we show that this is justified for finitely $C^\infty$-generated algebras. Moreover, in the affine case  the envelope inherits the Hopf algebra structure. In particular, this holds for finitely generated universal enveloping algebras and the quantum $SL(2)$.

Our approach is similar to that in non-commutative generalization of algebras of holomorphic functions given by Pirkovskii  who introduced \emph{holomorphically finitely generated algebras} --- non-commutative generalizations of Stein algebras --- as quotients of finitely generated algebras of free holomorphic functions \cite{Pi14, Pi15}. Thus we follow the same route but in the case of $C^\infty$ functions. 

Finitely $C^\infty$-generated algebras are a non-commutative generalization of $C^\infty$-differentiable algebras (Proposition~\ref{cifgcial}). Recall that the spectrum of a $C^\infty$-differentiable algebra is an affine $C^\infty$-differentiable space in the sense of Spallek whose theory is one of possible extensions of differential  geometry to spaces possibly admitting singularities \cite{NaSa}. Thereby $C^\infty$-generated algebras can be treated as an object of study in non-commutative geometry. However, it is not yet clear how they relate to NC differential geometry \`{a} la Connes but there is at least one  point in common: both theories use $C^\infty$ functional calculus.

Since the envelopes of finitely generated commutative real algebras are $C^\infty$-differentiable algebras, which have a geometric interpretation as mentioned above, we expect the envelope functor to be an avatar of a geometric functor similar to the analytization. But the corresponding theory in the $C^\infty$-case is still not well developed. Only a few hints can be found in the unpublished notes \cite[\S\,3.10]{Gi13} and the preprint \cite[\S\,2.6]{GM15}. Surely this issue deserves a separate study.

\subsection*{Algebras of polynomial growth}
Operators (and elements of Banach algebras) of polynomial growth is a classical topic \cite{CF68,LN00}. Banach algebras  of polynomial growth (i.e., such that all the elements are of polynomial growth) are introduced in \cite{ArOld}. Note that this assumption is quite restrictive, e.g., the ground field must be~$\R$. It is also proved in \cite[Theorem 2.8 and Proposition 2.9]{ArOld} that  every Banach algebra of polynomial growth is commutative modulo Jacobson radical and the radical is nilpotent.  Nevertheless, as  the author shows in \cite{ArOld,Ar22,ArNew}, the corresponding theory is rich enough to become an interesting field of research.

The next step is to move on to algebras that can be approximated by  Banach algebras  of polynomial growth  ($\mathsf{PGL}$ algebras). The prototypical example of a $\mathsf{PGL}$ algebra is $C^\infty(\R^k)$, where $k\in\N$. The most important feature of a $\mathsf{PGL}$ algebra is that every its element admits a $C^\infty$-functional calculus; see \cite[Theorem 3.2]{ArOld}.

\subsection*{Envelopes}
The most direct way to obtain a $\mathsf{PGL}$ algebra is to apply an enveloping functor.
The notion of an \emph{envelope with respect to a given class of Banach algebras} was introduced in \cite{ArNew} but, in an implicit form, it appeared earlier in \cite{ArOld} for the class of Banach algebras  of polynomial growth. Our study is based on envelopes with respect to this class.  It is also noted in \cite{Ar22B} and \cite{ArNew} that there are other interesting classes. Here is an extended list:

---  all Banach algebras;

---  Banach algebras of polynomial growth (notation is $\mathsf{PG}$);

--- strictly real Banach algebras in the sense of~\cite{In64};

---  Banach algebras satisfying a polynomial identity \cite{Ar22B};

--- Dedekind-finite Banach algebras (see Remark~\ref{Dede});

--- operator algebras (possibly non-selfadjoint closed subalgebras of the algebra of bounded operators in a Hilbert space).

(Note that the envelope with respect to the class of operator algebras is not the same as the enveloping operator algebra of a Banach algebra as described, e.g., in  \cite[2.4.6]{BlMe}).

If a class $\mathsf{C}$ of Banach algebras satisfies some natural conditions, the corresponding category  $\mathsf{CL}$ of projective limits admits a tensor algebra construction; see \cite[\S\,1]{ArNew}. Here we also consider other universal constructions  in $\mathsf{CL}$: symmetric algebras, coproducts, colimits and tensor products. In particular, we may put $\mathsf{C}=\mathsf{PG}$.

\subsection*{Finitely $C^\infty$-generated algebras}

In~\cite{ArNew}, algebras of free $C^\infty$-functions are introduced and a structure theorem in the case of finite number (say, $k$) of generators is announced for general $k$ and proved when $k\le 2$. The proof in the case of general finite $k$ is postponed to this text and the reader can find it in Appendix~\ref{sec:efa}. Here we are mainly interested in finitely $C^\infty$-generated algebras, which form a special subclass of $\mathsf{PGL}$ algebras.

We give some general theory of finitely $C^\infty$-generated algebras in \S\,\ref{s:finCial} and  show that this class is stable under finite projective tensor product (Theorem~\ref{tensfing}) --- the fact that enables us to introduce not only associative but also finitely $C^\infty$-generated topological Hopf algebras; see the next section. Some examples have already appeared in~\cite{ArNew}. Here we also consider  a new example, a $C^\infty$-version of a deformation of the two-dimensional non-abelian real Lie algebra~$\fa\ff_1$; see \S\,\ref{s:af1}.

\subsection*{Hopf algebras}
Recall that the main axiom of Hopf algebra states that the following diagram commutes:
\begin{equation*}
  \xymatrix{
 H\mathbin\otimes H\ar[d]^{S\mathbin\otimes 1}&\ar[l]_{\De}H\ar[d]^{u\varepsilon}\ar[r]^{\De}
 &H\mathbin\otimes H\ar[d]^{1\mathbin\otimes S}\\
H\mathbin\otimes H\ar[r]_m&H&H\mathbin\otimes H\ar[l]^m \,.}
\end{equation*}
Here $\De$ is the comultiplication, $m$ is the multiplication, $u\!:k\to H$ is the unit,
$\varepsilon\!:H\to k$ is the counit, $S\!:H\to H$ is the antipode, and $k$ is the ground field. In fact, $H\mathbin\otimes H$ here denotes  two different objects, an algebra when it is the codomain of $\De$ and a linear space when it is the domain of~$m$. They coincide, but it is a happy coincidence  that allows  this axiom to be formulated.

Turning to functional analysis, note that, for example, in the theory of $C^*$-algebraic quantum groups there is a need to work with two different tensor products, the Haagerup and the minimal. The difference between them leads to difficulties in definitions; see \cite{VvD01}. Fortunately, in a more general context, such a problem rarely occurs.  Indeed, the category of complete locally convex spaces (over $\CC$ or $\R$) is a~monoidal category  with respect to the bifunctor $(-)\ptn(-)$ of complete projective tensor product. A Hopf algebra in this category  is called a~\emph{Hopf $\ptn$-algebra}. This notion is well defined for the same reason as in the algebraic case: the complete projective tensor product of two associative $\ptn$-algebras is again an associative  $\ptn$-algebra. Various aspects of the theory of Hopf $\ptn$-algebras  are discussed in  \cite{BFGP,Chev, Li78,Pir_stbflat,PS98,HWa24} and also in the author's papers \cite{Ar21,Ar22B,ArDAAF,AHHFG, Ar24}.

If we restrict ourselves to the subclass consisting of Arens--Michael algebras (projective limits of Banach algebras) and want to endow some Arens--Michael algebra~$H$ with a structure of topological Hopf algebra (in a certain sense), then it is natural to assume that the codomain of the comultiplication is the Arens--Michael tensor product $H\mathbin{\otimes^{\,\mathsf{AM}}\!}H$. Since this algebra is isomorphic to $H\ptn H$ (see  Remark~\ref{exdeCitp}(B)), there is no need to introduce a new concept and go beyond Hopf $\ptn$-algebras. This implies that the Arens--Michael enveloping functor preserves Hopf $\ptn$-algebra structure; see \cite[Proposition 6.7]{Pir_stbflat}. Thus, it seems natural to study the subcategory of Arens--Michael Hopf algebras.

If we further narrow the class of Arens--Michael algebras and consider an HFG (stands for `holomorphically finitely generated') algebra~$H$ in the sense of Pirkovskii \cite{Pi14,Pi15} as underlying for a Hopf $\ptn$-algebra, then we can use the fact  that the codomain of the comultiplication, $H\ptn H$ is also HFG. So we get another natural subclass,  HFG Hopf algebras. They are studied in detail in~\cite{AHHFG}.

Considering $\mathsf{PGL}$  algebras instead of Arens--Michael and HFG algebras, we see that the situation is more subtle. The author does not know whether the enveloping functor with respect to $\mathsf{PG}$ preserves Hopf $\ptn$-algebra structure. This is because that we are only able to prove stability under $(-)\ptn(-)$  for the possibly smaller class $\mathsf{QDTS}$ but not for the whole $\mathsf{PGL}$; see Remark~\ref{PSDTvsPGL} and Theorem~\ref{prliCitp}. Of course, one can formally define $\mathsf{PGL}$ Hopf algebras, but at the moment this definition does not seem natural. On the other hand, Theorem~\ref{tensfing} asserts that the smaller class of finitely $C^\infty$-generated algebras is nevertheless stable under $(-)\ptn(-)$. Thus it is taken for granted to introduce a notion parallel to that of HFG Hopf algebra. We say that a real Hopf $\mathbin{\widehat{\otimes}}$-algebra finitely $C^\infty$-generated as a $\mathsf{PGL}$  algebra is  a \emph{$C^\infty$-finitely generated Hopf algebra}. We discuss this topic in \S\,\ref{s:Hopf}.

\subsection*{Contents}
In \S\,\ref{s:env} we recall some facts from \cite{ArNew} and also prove  new results on envelopes with respect to a class of Banach algebras. \S\,\ref{s:PGL} is devoted to Banach algebras of polynomial growth and $\mathsf{PGL}$ algebras. In \S\,\ref{s:uncon} we discuss some universal constructions in the category of projective limits of Banach algebras belonging to a certain class, namely, the tensor and symmetric algebra, coproduct, colimit and tensor product. Main properties of finitely $C^\infty$-generated algebras are proved in \S\,\ref{s:finCial}. A non-trivial example, a deformation of the two-dimensional non-abelian real Lie algebra~$\fa\ff_1$, is  contained in \S\,\ref{s:af1}. Finitely $C^\infty$-generated Hopf algebras are discussed in \S\,\ref{s:Hopf}. To prove results on projective products we use an explicit description of  $\cF_{k}^{\,\mathsf{PG}}$, the algebra of free $C^\infty$-functions of finite rank~$k$. This description is announced in~\cite{ArNew} but the proof is only given  in the case when $k\le2$. The argument is technical and we place the proof in the general case in Appendix~\ref{sec:efa}.


\subsection*{Acknowledgment}
A  part of this work was done during the author's visit to the HSE University (Moscow) in the summer of 2024.  I wish to thank this university
 for the hospitality.

\section{Envelopes}
\label{s:env}

In this section we discuss envelopes, a concept that formalizes the procedure of approximating a topological algebra by Banach algebras of a certain class.

Note that we consider topological algebras with separately continuous multiplication.  Also, all algebras and homomorphisms  are assumed to be unital.
The category of topological algebras is denoted by $\mathsf{TA}$.

\begin{df}\label{loccond} \cite[Definition~1.1%
]{ArNew}
Let  $\mathsf{C}$ be a class of Banach algebras (over $\CC$ or $\R$). We denote by $\mathsf{CL}$ the class of topological algebras isomorphic to a projective limit of algebras contained in~$\mathsf{C}$.  We also use notation $\mathsf{CL}$   for the corresponding full subcategory of $\mathsf{TA}$. If $A\in \mathsf{CL}$ we say that $A$ is \emph{locally in $\mathsf {C}$} or a \emph{$\mathsf{CL}$ algebra}.
\end{df}

\begin{df}\label{defengen} \cite[Definition~1.2%
]{ArNew}

(A)~Let $\mathsf{C}$ be a class of Banach algebras, $\mathsf{CL}$  defined as in Definition~\ref{loccond} and $F\!:\mathsf{CL}\to \mathsf{TA}$ the corresponding forgetful functor. If $F$ admits a left adjoint functor~$L$, we call $L$ the \emph{enveloping functor with respect to $\mathsf{C}$}.

(B)~Denote by $\io_A$ the component of the identity adjunction (the functor morphism $\id_{\mathsf{TA}}\Rightarrow F\circ L$) corresponding to a topological algebra~$A$  and put $\wh A^{\,\mathsf{C}}\!:=(F\circ L)(A)$. We say that the pair $(\wh A^{\,\mathsf{C}}, \io_A)$ is  the \emph{envelope of~$A$ with respect  to $\mathsf{C}$}.
\end{df}

In the case when $\mathsf{C}$ is the class of all Banach algebras we get the standard definition of the Arens--Michael envelope; see \cite{X1}.

\begin{pr}\label{exisenvgen}  \cite[Proposition~1.4%
]{ArNew}
Let $\mathsf{C}$ be a class of Banach algebras stable under passing to finite products and closed subalgebras. Then the enveloping functor with respect to~$\mathsf{C}$ exists.
\end{pr}

For efficient computation of envelopes, an additional  stability condition is useful.

\begin{pr}\label{PGLqugen}
Let $\mathsf{C}$ be a class of Banach algebras stable under passing to closed subalgebras and quotients (over closed ideals). If an algebra is locally in $\mathsf{C}$, then so is the completion of each of its quotient over a closed ideal.
\end{pr}

Note that in this text `ideal' stands for 'two-sided ideal'.

For the proof of Proposition~\ref{PGLqugen} we need a lemma. Recall first that the quotient of a complete locally convex space over a closed subspace is complete only under extra assumptions, e.g., in the metrizable case. The completion of a topological algebra $A$ is denoted by $\widetilde A$ or $A\sptilde$.

\begin{lm}\label{proquo}
Let $A$ be a projective limit of a system $(A_\nu)$ of Banach algebras, $(\tau_\nu\!:A\to A_\nu)$ the corresponding projective cone, and $I$ a closed ideal in~$A$. If the  range of $\tau_\nu$ is dense in  $A_\nu$ for every $\nu$, then $(A/I)\sptilde$ is the projective limit of the naturally defined directed system $(A_\nu/\,\overline{\tau_\nu (I)}\,)$ of Banach algebras.
\end{lm}
\begin{proof}
The density implies that, for every $\nu$, $\overline{\tau_\nu (I)}$  is a closed ideal in $A_\nu$  and so $A_\nu/\,\overline{\tau_\nu (I)}$ is a Banach algebra. When $\nu'\succeq \nu$, the linking homomorphism maps $\overline{\tau_{\nu}' (I)}$  into $\overline{\tau_\nu (I)}$ and thus there is a continuous homomorphism $(A_{\nu'}/\,\overline{\tau_{\nu'} (I)}\,)\to (A_\nu/\,\overline{\tau_\nu (I)}\,)$. It is clear that we obtain  a projective system $(A_\nu/\,\overline{\tau_\nu (I)})$ of Banach algebras.

The density also implies that the projective limit provides an Arens--Michael decomposition of~$A$. Then we can use \cite[Lemma 7.1]{Pi09}, which immediately implies that $(A/I)\sptilde$ is a projective limit of $(A_\nu/\,\overline{\tau_\nu (I)}\,)$.
\end{proof}

\begin{proof}[Proof of Proposition~\ref{PGLqugen}]
Let $A$ be a projective limit of a system $(A_\nu)$ such that $A_\nu\in\mathsf{C}$ for every $\nu$ and $I$ a closed ideal in~$A$. Since $\mathsf{C}$ is stable under passing to closed subalgebras, we can assume that every homomorphism $A\to A_\nu$ in the corresponding cone has dense range.  By Lemma~\ref{proquo}, the completion of $A/I$  is the projective limit of the system $(A_\nu/\,\overline{\tau_\nu (I)}\,)$. Since $\mathsf{C}$ is stable under passing to quotients, all the algebras in this system are in~$\mathsf{C}$. Thus $(A/I)\sptilde$ is locally in $\mathsf{C}$.
\end{proof}

In view of Propositions~\ref{exisenvgen} and~\ref{PGLqugen}, the following result is deduced in exactly  the same way as in \cite[Proposition 6.1]{Pir_stbflat} and so we omit the proof.
\begin{pr}
\label{PGquot}
Let $\mathsf{C}$ be a class of Banach algebras stable  under passing to finite products, closed subalgebras and quotients. Suppose that~$A$ is a topological algebra and~$I$ is a (possibly non-closed) ideal in~$A$. Denote the closure of the image of~$I$ in $\wh{A}^{\,\mathsf{C}}$ by~$J$. Then $J$ is an ideal in $\wh{A}^{\,\mathsf{C}}$ and the homomorphism $A/I\to \wh{A}^{\,\mathsf{C}}/J$ generated by $\iota_A\!:A\to\wh{A}^{\,\mathsf{C}}$ extends to an isomorphism
$$
\wh{A/I}^{\,\mathsf{C}}\cong (\wh{A}^{\,\mathsf{C}}/J)\sptilde
$$
of topological algebras. In particular, if $\wh{A}^{\,\mathsf{C}}$ is a Fr\'echet algebra, then $\wh{A/I}^{\,\mathsf{C}}\cong \wh{A}^{\,\mathsf{C}}/J$.
\end{pr}

\begin{rem}\label{Dede}
All the classes of Banach algebras listed in the introduction are stable under passing to finite products and closed subalgebras. They are also stable under passing to quotients except for Dedekind-finite Banach algebras. So we cannot apply Proposition~\ref{PGquot} to algebras in the last class.

Indeed, recall that a ring is said to be \emph{Dedekind-finite}, if  $ab=1$ implies $ba=1$. In the context of $C^*$-algebras, a Dedekind-finite algebra is simply called  'finite'; see, e.g., \cite[Exercise~5.1 and Lemma 5.1.2]{RLL}. It immediately follows from the definitions that the class of Dedekind-finite Banach algebras is stable under passing to finite products and closed subalgebras and so the corresponding envelope exists by Proposition~\ref{exisenvgen}.

On the other hand,  to show that the class of Dedekind-finite Banach algebras is not stable under passing to quotients it suffices to show that the class of finite $C^*$-algebras is not stable. In the case of the ground field $\CC$, this fact is well known and furthermore there are many $C^*$-algebras that are  quotients of finite  but not  themselves finite. For example, every nuclearly  imbeddable unital separable $C^*$-algebra  $A$ is a quotient of a $C^*$-subalgebra of the CAR-algebra \cite[Corollary 1.4, (ii)$\Leftrightarrow$ (v)]{Ki95}; see also \cite[Theorem~2]{Wa94}. Since the CAR-algebra is an AF-algebra, it is finite \cite[Theorem IV.2.3]{Da96}. It follows from the equivalent definition of finite $C^*$-subalgebra in terms of projections \cite[Definition 5.1.1]{RLL} that each of its $C^*$-subalgebras is finite. But can additionally suppose that $A$ is infinite and thus the quotient property does not hold.
\end{rem}

\section{$\mathsf{PGL}$ and $\mathsf{DTS}$ algebras}
\label{s:PGL}

In this section we consider the case when $\mathsf{C}=\mathsf{PG}$ and prove some preliminary results on $\mathsf{PGL}$ algebras.

\subsection*{Banach algebra  of  polynomial growth}

The following definition is from \cite{ArOld}.

\begin{df}
\label{defPG}
(A)~An element $b$ of a Banach algebra~$B$ is said to be  of \emph{polynomial growth} if there are $K>0$ and $\alpha\ge0$ such that
\begin{equation*}
\|e^{isb}\|\le K (1+|s|)^{\alpha} \quad \text{for all }s\in \mathbb{R}.
\end{equation*}
(In the case of ground field $\R$ we extends the norm to the complexification of~$B$.)

(B)~A real Banach algebra is said to be of  \textit{polynomial growth} if all its elements are of polynomial growth.
\end{df}

For $p\in\N$ denote by $\rT_p$ the algebra of upper triangular real matrices of order~$p$. When considering a finite-dimensional associative algebra (e.g., $\rT_p$), we always assume that some submultiplicative norm is fixed and so we get a Banach algebra.

\begin{pr}\label{PGfd}
A finite-dimensional Banach algebra is of polynomial growth if and only if it is isomorphic to a subalgebra
of $\rT_p$ for some $p\in\N$.
\end{pr}
\begin{proof}
To prove the necessity suppose that $A$ is a finite-dimensional Banach algebra of polynomial growth.
Denote $A$ with the Lie algebra structure given by $[a,b]\!:=ab-ba$ by $\mathrm{Lie}(A)$.
Since $A$ is of polynomial growth and  finite dimensional, $\mathrm{Lie}(A)$ is triangular by \cite[Proposition~4.1]{ArOld}. It follows from a generalization of Lie’s theorem \cite[\S\,1.2, p.\,10, Theorem 1.2]{VGO90} that every finite-dimensional real representation of $\mathrm{Lie}(A)$ is  triangular, i.e.,  there is a complete flag of invariant subspaces. Taking a faithful representation of~$A$ (e.g., the regular one), considering it as representation of $\mathrm{Lie}(A)$ and choosing a basis compatible with a complete flag, we have that $A$ is isomorphic to a subalgebra of $\rT_p$, where  $p$ is the dimension of the representation.

On the other hand, it is not hard to see that $\rT_p$ is of polynomial growth for every $p$ (for example, one can apply \cite[Theorem~2.14]{ArOld}) and so are all its subalgebras.
\end{proof}

Recall that a \emph{quiver} is a directed graph, possibly with  loops and multiple
arrows. Denote by $\R Q$ the path algebra (over $\R$) of a quiver $Q$; for details see, e.g. \cite[\S\,4.1]{Be91}.  The following proposition is used in the proof of  Theorem~\ref{multCiffk}.

\begin{pr}\label{patalnc}
Let $Q$ be a quiver with finitely many vertices and arrows, and
has no oriented cycle. Then $\R Q$ is finite dimensional algebra of polynomial growth.
\end{pr}
\begin{proof}
It is easy to see that $\R Q$ is finite dimensional.

To show that $\R Q$ has polynomial growth denote by $A$ the linear span of all idempotents in $\R Q$ corresponding to the paths of length~$0$ (i.e., the vertices) and   by $I$ the linear span of the paths of non-zero length. It is easy to see that $A$ is a subalgebra and $I$  is a nilpotent ideal. Note that $\R Q=A+I$ and so we obtain a split nilpotent extension $0\leftarrow A \leftarrow \R Q\leftarrow I\leftarrow 0$. Since $A$ is of polynomial growth, it follows from \cite[Theorem~2.14]{ArOld} that so is~$\R Q$.
\end{proof}

Now we discuss tensor products.
Recall that the complete projective tensor product $A\ptn B$ of Banach algebras $A$ and $B$ is also a Banach algebra with respect to naturally defined multiplication.

\begin{rem}\label{tprFou}
Generally speaking, Banach algebras of polynomial growth are poorly compatible with the projective tensor product.
For example, although $C(\T)$ is of polynomial growth, $C(\T)\ptn C(\T)$ is not. Indeed, by Varopoulos's theorem, the projective tensor product $C(\T)_\CC\ptn C(\T)_\CC$ of the complexifications contains the algebra~$A(\T)$  of absolutely
convergent Fourier series as a closed subalgebra; see \cite[Chapter VIII, \S\,3]{Ka70}. Denote by $A(\T)_\R$ the subalgebra of real-valued functions in $A(\T)$. Then we can treat it as  a closed subalgebra of $C(\T)\ptn C(\T)$.

If $C(\T)\ptn C(\T)$ had polynomial growth, so would $A(\T)_\R$. But this contradicts a theorem  of Katznelson \cite[Chapter VI, \S\,6]{Ka70}. For details see \cite[Remark 2.7]{ArOld}.
\end{rem}

The following result is a strengthening of Theorem 2.12 in~\cite{ArOld}.

\begin{thm}\label{polfrfd}
Let $A$ and $B$ be  Banach algebras of polynomial growth. If one of them is finite dimensional, then
$A\ptn B$ is of polynomial growth.
\end{thm}
\begin{proof}
The argument is essentially the same as in the case when $A=\rT_{p}$ and $B$ is commutative; see \cite[Theorem~2.12]{ArOld}. It follows from \cite[Theorem~2.14]{ArOld} that it suffices to show that $A\ptn B$ is a split nilpotent extension of an algebra of polynomial growth.

Denote by $\Rad A$ the Jacobson radical of $A$ and consider the extension
\begin{equation}\label{Asptex}
0\leftarrow A/\Rad A \leftarrow A \leftarrow \Rad A\leftarrow 0.
\end{equation}

Since $\R$ is a perfect field and $A$ is finite dimensional, we can apply the Wedderburn factor theorem \cite[\S\,13, p.\,471--472, 13.18 and 13.19]{Fa73}, which implies that \eqref{Asptex} splits and $A/\Rad A$  is  classically  semisimple. It follows from the Wedderburn decomposition theorem \cite[Theorem~2.2]{Kna} that a classically  semisimple real algebra is isomorphic to a finite product of full matrix algebras over division algebras over $\R$. Moreover, by
the Frobenius  theorem \cite[Theorem~2.50]{Kna}, the only finite-dimensional division $\R$-algebras are $\R$, $\CC$ and quaternions. Thus, since $A/\Rad A$ is commutative by \cite[Theorem~2.8]{ArOld}, it is a finite product of copies of $\R$ and $\CC$.

It follows from \cite[Proposition~2.3]{ArOld} that, being a quotient of an algebra of polynomial growth,  $A/\Rad A$ is also of polynomial growth and hence strictly real (the spectrum of each element contains in $\R$) by \cite[Proposition~2.5]{ArOld}. Therefore $A/\Rad A$ contains no copy of $\CC$ and so is isomorphic to~$\R^k$ for some $k\in\N$.

Further, $\Rad A$ is nilpotent because $A$ is finite dimensional; see, e.g., \cite[Proposition~I.3.59]{X2}. Since \eqref{Asptex} is splits and nilpotent, so is the extension
$$
0\leftarrow B^k \leftarrow A\ptn B \leftarrow (\Rad A)\ptn B\leftarrow 0
$$
obtained by tensoring with $B$. Since  $B$ is of polynomial growth, so is $B^k$  \cite[Proposition~2.13]{ArOld}. An application of Theorem~2.14 in \cite{ArOld} mentioned above completes the proof.

 \end{proof}

Let  $A$ be a real Banach algebra, $m$ in $\N$ and $K$ a compact subset of $\R^m$ with dense interior. Denote by $C^n(K,A)$ the Banach algebra of $A$-valued functions on $K$ having continuous derivatives of order~$n$. Put $C^n(K)\!:= C^n(K,\R)$. Since $C^n(K)\ptn A \cong C^n(K,A)$ when $A$ is finite dimensional, and $C^n(K)$ is of polynomial growth \cite[Proposition 2.13]{ArOld}, we immediately obtain the following corollary of Theorem~\ref{polfrfd}.

\begin{co}\label{CnKA}
Suppose that $A$ is a finite-dimensional real associative algebra, $m\in\N$, $n\in\Z_+$ and $K$ is a compact subset of $\R^m$ with dense interior. If $A$ is of polynomial growth, then so is $C^n(K,A)$.
\end{co}

The following result has no use in the rest of this paper but it is of independent interest.

\begin{pr}\label{tenprPG}
Let $\fA$ and $\fA'$ be split nilpotent extensions of Banach algebras $A$ and~$A'$, respectively.

\emph{(A)}~Then $\fA\ptn \fA'$ is a split nilpotent extension of $A\ptn A'$.

\emph{(B)}~If, in addition, $A\ptn A'$ is of polynomial growth,  then  so is $\fA\ptn \fA'$.
\end{pr}
\begin{proof}
(A)~Let
\begin{equation}\label{twoext}
0\lla A \lla \fA \lla I\lla 0\quad\text{and}\quad 0\lla A' \lla \fA' \lla I'\lla 0
\end{equation}
be the corresponding split nilpotent extensions. Then $I$ and $I'$ are complemented closed subspaces of  $\fA$ and $\fA'$, respectively. Therefore we can identify $I\ptn \fA'+\fA\ptn I'$ with an ideal of $\fA\ptn \fA'$ complemented as a closed subspace.

It is easy to see that
$$
0\lla A\ptn A' \lla \fA\ptn\fA' \lla I\ptn \fA'+\fA\ptn I'\lla 0,
$$
splits since the extensions in~\eqref{twoext} split. Moreover, $I\ptn \fA'+\fA\ptn I'$ is nilpotent since
$I$ and $I'$ are nilpotent.

By \cite[Theorem~2.14]{ArOld}, Part~(B) follows immediately from Part~(A).
\end{proof}

\subsection*{$\mathsf{PGL}$  algebras}

The following notation is from \cite{ArNew}.

\begin{nota}
We denote by $\mathsf{PG}$ the class of real Banach algebras of polynomial growth and by $\mathsf{PGL}$
the class of algebras that are locally in $\mathsf{PG}$. (Sometimes instead of `locally in $\mathsf{PG}$' we write 'a $\mathsf{PGL}$ algebra'.)
\end{nota}

\begin{pr}\label{PGLqu}
The completion of every quotient algebra of a $\mathsf{PGL}$ is also in~$\mathsf{PGL}$.
\end{pr}
\begin{proof}
The class $\mathsf{PG}$ is stable under passing to closed subalgebras \cite[Proposition~2.11(A)]{ArOld} and quotients \cite[Proposition~2.3]{ArOld}. So the result follows from Proposition~\ref{PGLqugen}.
\end{proof}

The following corollary is a partial case of Proposition~2.8
 in \cite{ArNew} but we include the proof for completeness.
\begin{co}\label{maniPGL}
Let $M$ be a Hausdorff smooth manifold with countable base. Then $C^\infty(M)$ is a $\mathsf{PGL}$ algebra.
\end{co}
\begin{proof}
We can assume that $M$ is a closed submanifold in $\R^m$ for some~$m$.  The corresponding continuous homomorphism $C^\infty(\R^m)\to C^\infty(M)$ is surjective; see, e.g., \cite[\S\,2, p.\,24, Proposition 2.4]{NaSa}. Since both algebras are Fr\'echet spaces, the open mapping theorem implies that  $C^\infty(M)$ is a quotient of $C^\infty(\R^m)$. Since
$C^\infty(\R^m)$ is a $\mathsf{PGL}$ algebra \cite[Theorem 2.12]{ArOld}, then so is $C^\infty(M)$ by Proposition~\ref{PGLqu}.
\end{proof}

This result extends to a larger class of algebras. Namely, following \cite[\S\,2.4, p.\,30]{NaSa},
we say that a \emph {$C^\infty$-differentiable algebra} is a real Fr\'echet--Arens--Michael algebra that is topologically isomorphic to an algebra of the form $C^\infty(\R^m )/J$, where $m\in\N$ and $J$~is a closed ideal. (The definition given in~\cite{NaSa} requires only an algebraic isomorphism because the topology on a $C^\infty$-differentiable algebra is independent of its representation as a quotient \cite[p.\,34, Theorem 2.23]{NaSa}.) From Proposition~\ref{PGLqu} we immediately get the following result.

\begin{co}\label{CdiffPGL}
Every $C^\infty$-differentiable algebra is in $\mathsf{PGL}$.
\end{co}

The following result is an immediate consequence of Proposition~\ref{PGquot} in the case when $\mathsf{C}=\mathsf{PG}$. It shows that the notion of envelope with respect to  $\mathsf{PG}$ has a natural interpretation in the context of differential geometry and its extension to spaces with singularities.

\begin{pr}\label{fgCidiff}
Let $C$ be a finitely generated commutative real algebra. Then $\wh C^{\,\mathsf{PG}}$ is a $C^\infty$-differentiable algebra.
\end{pr}

\begin{rem}
In view of Corollary~\ref{fgCidiff}, we can treat the envelope with respect to $\mathsf{PG}$ of a general commutative real algebra as the algebra of $C^\infty$-functions (modulo nilpotent elements) on some commutative space. For example, $C^\infty(\R^k)$ is the envelope of $\R[x_1,\ldots,x_k]$ \cite[Proposition~3.3
]{ArNew}. Moreover, following the standard paradigm of non-commutative geometry, we can say that the envelope of a general non-commutative algebra is the algebra of $C^\infty$-functions on a non-commutative space. See an extension of this approach in \S\,\ref{s:finCial}.
\end{rem}

We now give two examples of applying of Proposition~\ref{PGquot} to finding the explicit form of the envelope.

\begin{exm}\label{exmcross}
Let  $A=\R[x,y]/(xy)$, where $(xy)$ is the ideal generated by~$xy$. Our purpose is to find $\wh{A}^{\,\mathsf{PG}}$.

We first claim that the homomorphism
$$
\rho\!:\R[x,y]\to \R[t]\times \R[t]\!:\rho(h)\!:= (x\mapsto h(x,0),\,y\mapsto h(0,y))
$$
induced an isomorphism
\begin{equation}\label{rhoprArOm}
\rho'\!:A\to \mathcal{R}(\Om),
\end{equation}
where $\mathcal{R}(\Om)\!:=\{(f,g)\in \R[t]\times \R[t]\!:\,f(0)=g(0) \}$. (Here $\Om$ stands for the set of characters, which coincides with the union of the coordinate axes but we do not need this fact here.)

Indeed, $\rho$ maps $(xy)$ to $0$ and so $\rho'$ is well defined. If $\rho(h)=0$, then $h$ is divisible by~$xy$ and so it is in $(xy)$. Thus $\rho'$ is injective. On the other hand, if $f,g\in \R[t]$ and $f(0)=g(0)$, then we have $\rho(h)=(f,g)$ when $h(x,y)=f(x)+g(y)-f(0)$. It follows that $\rho'$ is surjective and hence bijective.

We now claim that $\wh{A}^{\,\mathsf{PG}}$ is given by the same restriction, i.e., the continuous extension of $\rho$ to a homomorphism $C^\infty(\R^2)\to C^\infty(\R)\times C^\infty(\R)$ induces a topological isomorphism $\rho'\!:\wh{A}^{\,\mathsf{PG}}\to C^\infty(\Om)$, where
\begin{equation*}
 C^\infty(\Om):=\{(f,g)\in C^\infty(\R)\times C^\infty(\R)\!:\,f(0)=g(0) \}.
\end{equation*}
(In fact, $\rho'$ is an extension of the map in~\eqref{rhoprArOm}.)

Denote by $J$ the ideal in $C^\infty(\R^2)$ generated by the polynomial $xy$.
Since $J$ is closed (see \cite[Proposition~3.9
]{ArNew}), it coincides with the closure of $(xy)$. Proposition~\ref{PGquot} (in the case when $\mathsf{C}=\mathsf{PG}$) implies that $\wh{A}^{\,\mathsf{PG}}\cong C^\infty(\R^2)/J$.
Since $xy$ is in the kernel of $\rho$, we have that $\rho'$ is well defined.

To prove that $\rho'$ is injective, we check that $xy$ satisfies the \emph{zero property}, i.e., all functions in $C^{\infty}(\R^2)$ vanishing on $\Om$ belong to $J$. Indeed, it follows from Hadamard's lemma that the functions $x$ and $y$ satisfy the zero property. Since the origin is a single critical point of the function $xy$, the set of regular points of $\Om$ is dense. So by \cite[Corollaire~3]{BO73} the function $xy$ also satisfies the zero property. The argument for the surjectivity of $\rho'$ is 
the same as above.

Thus $\rho'$ is bijective. Since we deal with Fr\'echet spaces, the inverse mapping theorem implies that  $\rho'$ is a topological isomorphism.
\end{exm}

\begin{exm}\label{cups}
Let  $A=\R[x,y]/(x^2-y^3)$.   Our purpose is to find $\wh{A}^{\,\mathsf{PG}}$ in this case.

Denote also by $J$ the ideal in $C^\infty(\R^2)$ generated by $x^2-y^3$. Using the parametrization $x=t^3$, $y=t^2$, we can identify $A$ with the subalgebra in $\R[t]$ consisting of polynomials with zero coefficient at the  term of first degree, i.e., with trivial first derivative at~$0$. We now show that the same holds for the envelope, i.e., $\wh{A}^{\,\mathsf{PG}}\cong C^\infty_1$, where 
$$
C^\infty_1\!:=\{f\in C^\infty(\R)\!:\, f'(0)=0\}.
$$

In this case, denote by $\Om$ the cuspidal cubic determined by the equation $x^2=y^3$. Being finitely generated, $J$ is closed; see \cite[Proposition~3.9
]{ArNew}. By Proposition~\ref{PGquot}, we have that $\wh{A}^{\,\mathsf{PG}}\cong C^\infty(\R^2)/J$.

Consider the continuous homomorphism
$$
\rho\!:C^\infty(\R^2)\to C^\infty(\R)\!:g\mapsto (t\mapsto g(t^3,t^2)).
$$
It is clear that the range of $\rho$ is contained in $C^\infty_1$ and $\rho(J)=0$. So we have the continuous homomorphism
$$
\rho'\!:\wh{A}^{\,\mathsf{PG}}\to C^\infty_1.
$$
We need to show that $\rho'$ is a topological isomorphism.

(1)~To prove that $\rho'$ is injective it suffices to show that $p(x,y)\!:=x^2-y^3$ satisfies the zero property in $C^\infty(\R^2)$.
The following argument is from the slides~\cite{K017} of Kondo. Let $\psi\in C^\infty(\R^2)$ and vanish on $\Om$. Since
$$
p(0,0)=\frac{\partial p}{\partial x}(0,0)=0\quad \text{and}\quad \frac{\partial^2 p}{\partial x^2}(0,0)\ne 0,
$$
it follows from the Malgrange preparation  theorem \cite[Theorem 7.5.6]{Ho76} that there are $q\in C^\infty(\R^2)$ and $h_0,h_1\in C^\infty(\R)$ such that
$$
\psi(x,y) = q(x,y)p(x,y)+h_1(y)x + h_0(y)\qquad(x,y\in\R).
$$
If $y\ge 0$, then $(y^{3/2},y)$ and $(-y^{3/2},y)$ belong to $\Om$. Since $\psi$ and $p$ vanish on $\Om$, we have that
$$
h_1(y)y^{3/2} + h_0(y)=-h_1(y)y^{3/2} + h_0(y)=0
$$
and so $h_0(y)=h_1(y)=0$. Hence the function $h(x,y)\!:= h_1(y)x + h_0(y)$ vanishes on the set $\{(x,y)\!:\,y\ge 0\}$ and therefore all the partial derivatives vanish on this set and on its subset $\Om$.

We claim that $h\in J$. Indeed, since $J$ is closed (see above), by Whitney’s spectral theorem \cite[Chapter~II]{Ma66}, it suffices to show that for every point $(x_0,y_0)$ of~$\R^2$ there is  a function in $J$ whose Taylor expansion at $(x_0,y_0)$ coincides with that of~$h$. Since $h$ vanishes in upper semiplane, this is obvious when $y_0\ge 0$. Otherwise, when $y_0<0$, we can take a function $q_1$ that equals $h(x,y)p(x,y)^{-1}$ in a small enough neighbourhood of $(x_0,y_0)$ and $0$ when $y\ge 0$. Then the Taylor expansion of $q_1 p$ at $(x_0,y_0)$ coincides with that of~$h$. Thus $h_1(y)x + h_0(y)=q_1(x,y)p(x,y)$, where $q_1\in C^\infty(\R^2)$. Hence $\psi\in J$, which means that $x^2-y^3$ has the zero property. Thus $\rho'$ is injective.

(2)~To prove that $\rho'$ is surjective take $f\in C^\infty_1$ and write $f=f_{+}+f_{-}$, where $f_{+}$ is even, $f_{-}$ is odd and both functions are in $C^\infty(\R)$. It is clear that $f'_{-}(0)=0$.

We claim first that if $f\in C^\infty(\R)$ and even, then the function $g(y)\!:= f(y^{1/2})$ is in
$C^\infty[0,+\infty)$. Indeed, it is obvious that $g$ is of class  $C^\infty$ for $y>0$.
In the case when $y=0$,  take $n\in\N$ and write $f(t)=P_{2n}(t)+o(t^{2n})$ as $t\to 0$, where $P_k$ denotes the Taylor polynomial of degree $k$ at $t=0$. Since $P_{2n}(y^{1/2})$ is a polynomial in~$y$ and $g(y)-P_{2n}(y^{1/2})=o(y^n)$ as $y\to 0+$,  we have that $g$ is $n$ times differentiable at~$0$. Since $n$ is arbitrary, $g\in C^\infty[0,+\infty)$.

We claim second that if $f$ is an odd function  in $C^\infty(\R)$ and $f'(0)=0$, then
the function $y\mapsto f(y^{1/2})/y^{3/2}$ extends to a function $\wt f \in C^\infty[0,+\infty)$. Indeed, since the Taylor's coefficients  at degrees $0$, $1$, $2$ are trivial,  the function  $t\mapsto f(t)/t^3$ extends to a function in $C^\infty(\R)$ and so we can apply the first claim.

Now for $f\in C^\infty_1$ put $g(x,y)\!:=f_+(y^{1/2})+x \wt f_-(y^{1/2})$. It follows from the claims above that $g$ is of class $C^\infty$ in the upper semiplane. So it can be extended to a function in $C^\infty(\R^2)$. It is obvious that $\rho(g)=f$.  Thus $\rho'$ is surjective and hence a topological isomorphism. Thus $\wh{A}^{\,\mathsf{PG}}\cong  C^\infty_1$.
\end{exm}

We now discuss projective tensor products of $\mathsf{PGL}$ algebra. Our considerations for the Banach algebra case can be extended to a more general class, which is (possibly strictly) smaller than $\mathsf{PGL}$. Note first that if $A$ and $B$ are Arens--Michael algebras (i.e., projective limits of Banach algebras), then so is $A\ptn B$ with respect to a naturally defined multiplication. Moreover, $A\ptn B$ satisfies a natural universal property; see Remark~\ref{exdeCitp}(B).

The following result is a strengthening  of Theorem~\ref{polfrfd}.

\begin{pr}\label{DfsTpg}
If $A$ is a $\mathsf{PGL}$ algebra and $B$ is a finite-dimensional algebra of polynomial growth, then $A\ptn B$ is also in $\mathsf{PGL}$.
\end{pr}
\begin{proof}
By \cite[Proposition 2.6]{ArNew}, an Arens--Michael $\R$-algebra is in $\mathsf{PGL}$ if only if it is isomorphic to a closed subalgebra of a product of Banach algebras of polynomial growth. So we can assume that $A\subset \prod_i A_i$, where each of $A_i$ is in $\mathsf{PG}$. Since the complete projective tensor product commutes with products of complete locally convex spaces (see, e.g., \cite[Chapter~II, p.\,127, Theorem~5.19]{X1}) and hence with products of complete locally convex algebras, $A\ptn B$ is isomorphic to a closed subalgebra of  $\prod_i (A_i\ptn B)$. By Theorem~\ref{polfrfd}, each of $A_i\ptn B$ is in $\mathsf{PG}$.  Applying \cite[Proposition 2.6]{ArNew} in the reverse direction, we conclude that $A\ptn B$ is in $\mathsf{PGL}$.
\end{proof}

\subsection*{$\mathsf{DTS}$ and $\mathsf{QDTS}$ algebras}
We introduce two more classes of algebras and show that they are contained in $\mathsf{PGL}$ and stable under tensor products, which is used in \S\,\ref{s:finCial}.

\begin{df}\label{PSDTde}
We denote by $\mathsf{DTS}$ the class of closed subalgebras of products of algebras of the form $C\ptn \rT_p$, where $C$ is a $C^\infty$-differentiable algebra and $p\in\N$. (Here `D' stands  for `differentiable', `T' for `triangular' and `S' for `subalgebra'.)
\end{df}

By Proposition~\ref{PGfd}, we can take arbitrary finite-dimensional algebras of polynomial growth instead of~$\rT_p$.

Note also that this definition is a bit redundant for our purposes since in fact we only use  closed subalgebras of products of algebras of the form $C^\infty(M)\ptn \rT_p$, where $M$ is a manifold (see \S\,\ref{s:af1} and~Appendix~\ref{sec:efa}). But it seems that the class $\mathsf{DTS}$ is of independent interest.

\begin{pr}\label{PSDTPGL}
Every algebra in $\mathsf{DTS}$ belongs to $\mathsf{PGL}$.
\end{pr}
\begin{proof}
By Corollary~\ref{CdiffPGL}, every $C^\infty$-differentiable algebra is in $\mathsf{PGL}$. Then by Proposition~\ref{DfsTpg}, every algebra of the form $C\ptn \rT_p$, where $C$ is a $C^\infty$-differentiable algebra, is in $\mathsf{PGL}$. Moreover, a closed subalgebra of a product of algebras that are in $\mathsf{PGL}$ is also in~$\mathsf{PGL}$ \cite[Theorem~2.7%
]{ArNew}. It follows from the definition that  $\mathsf{DTS}\subset\mathsf{PGL}$.
\end{proof}

We now turn to tensor products of $\mathsf{DTS}$ algebras.
The following result is used in \S\,\ref{s:finCial}.

\begin{pr}\label{prliCitp}
If $A_1$ and $A_2$ are in $\mathsf{DTS}$, then so is $A_1\ptn A_2$.
\end{pr}
\begin{proof}
(1)~Consider first the case when $A_i$ has the form $C_i\ptn B_i$, where $C_i$ is a $C^\infty$-differentiable algebra and $B_i$ is a finite-dimensional algebra of polynomial growth ($i=1,2$). By  \cite[Theorem~6.13]{NaSa}, $C_1\ptn C_2$ is a $C^\infty$-differentiable algebra and, by Theorem~\ref{polfrfd},  $B_1\ptn B_2$ is a finite-dimensional algebra of polynomial growth. Thus $A_1\ptn A_2$ also has the same form.

(2)~Suppose now that $A_1$ and $A_2$ are products of algebras of the form considered in Case~(1). Since the complete projective tensor product commutes with products of complete locally convex algebras (cf. the proof of proposition~\ref{DfsTpg}), it follows from Case~(1) that $A_1\ptn A_2$ is also a product of algebras of the above form.

(3)~Finally, consider  the general case.  Let $A_1\subset B_1$ and $A_2\subset B_2$, where $B_1$ and $B_2$ are algebras of the form considered in Case~(2). Note that $C^\infty$-differentiable algebras and finite-dimensional algebras are nuclear spaces. Moreover, the class of nuclear spaces is stable under  products \cite[Chapter~3, \S\,7.4]{Scha} and passing to closed subspaces. Hence all the space under consideration are nuclear. Therefore the natural homomorphism $A_1\ptn A_2\to B_1\ptn B_2$ is topologically injective; see, e.g., \cite[p.\,322, Theorem A1.6]{EP96}. Thus $A_1\ptn A_2$ is a closed subalgebra of an algebra of the form considered in Case~(2), i.e.,  $A_1\ptn A_2$ is in $\mathsf{DTS}$.
\end{proof}

\begin{df}\label{QPSDTde}
We denote by $\mathsf{QDTS}$ the class of completions of quotients of $\mathsf{DTS}$ algebras.
\end{df}

We need to complete quotients because  there are non-metrizable $\mathsf{DTS}$ algebras.

The next result follows immediately from Propositions~\ref{PSDTPGL} and \ref{PGLqu}.

\begin{pr}\label{QSDTPGL}
Every algebra in $\mathsf{QDTS}$ belongs to $\mathsf{PGL}$.
\end{pr}

The following proposition is similar to Proposition~\ref{prliCitp}.

\begin{pr}\label{QprliCitp}
If $A_1$ and $A_2$ are in $\mathsf{QDTS}$, then so is $A_1\ptn A_2$.
\end{pr}
\begin{proof}
It is easy to see that an Arens--Michael algebra is in $\mathsf{QDTS}$  if and only if it contains a dense image of a $\mathsf{DTS}$ algebra under a continuous homomorphism whose corestriction to this image is open.
Fix for $i=1,2$ such a homomorphism $\phi_i\!:B_i\to A_i$.

It is well known that the above property of linear maps is preserved by the projective tensor product; see, e.g., \cite[Proposition~43.9 and  Definition 4.1]{Tr67}. In our case this means that the corestriction of $\phi_1\otimes \phi_2$ to the image of $B_1\ptn B_2$  in $A_1\ptn A_2$  is open and the image is dense. Since $B_1\ptn B_2$ is in $\mathsf{DTS}$ by Propositions~\ref{prliCitp} and $\phi_1\otimes \phi_2$ is a homomorphism, $A_1\ptn A_2$ is also in $\mathsf{QDTS}$.
\end{proof}

\begin{rem}\label{PSDTvsPGL}
The author does not know whether $\mathsf{DTS}$, or at least $\mathsf{QDTS}$, coincides with $\mathsf{PGL}$ or not. It seems doubtful but it is possible that these classes are not much different from  $\mathsf{PGL}$.
\end{rem}

\subsection*{Commutativity modulo radical}

The next result, which is a strengthening of \cite[Theorem~2.8]{ArOld}, is also used in \S\,\ref{s:finCial}.

\begin{pr}\label{PGLcmR}
Every $\mathsf{PGL}$ algebra is commutative modulo the radical.
\end{pr}

For the proof we need the following lemma.

\begin{lm}\label{PrlmodR}
Suppose that a real or complex Arens--Michael algebra $A$ is a projective limit of Banach algebras commutative modulo the radical. Then $A$ is also commutative modulo the radical.
\end{lm}
\begin{proof}
Let $a,b\in A$. We need to show that $[a,b]\in\Rad A$, i.e.,  $1-c[a,b]$ is invertible for every $c\in A$.

Let $A$ be the projective  limit of a system $(A_\nu)$ of Banach algebras, each of which is commutative modulo the radical. For every $\nu$ we denote by $\tau_\nu$ the corresponding homomorphism $A\to A_\nu$. Then $\tau_\nu([a,b])=[\tau_\nu(a),\tau_\nu(b)]\in\Rad A_\nu$. Hence $1-x[\tau_\nu(a),\tau_\nu(b)]$ is invertible for every $x\in A_\nu$. In particular, $\tau_\nu(1-c[a,b])$ is invertible for every $c\in A$. However, an element $d \in A$ is invertible when all $\tau_\nu(d)$ are invertible; see e.g., \cite[p.\,280, Proposition V.2.11]{X2}. This completes the proof.
\end{proof}

\begin{proof}[Proof of Proposition~\ref{PGLcmR}]
It follows from \cite[Theorem 2.8]{ArOld} that every Banach algebra of polynomial growth is commutative modulo the radical. Since an algebra in $\mathsf{PGL}$ is, by definition, the projective limit of Banach algebras of polynomial growth, we conclude from Lemma~\ref{PrlmodR} that it is also commutative modulo the radical.
\end{proof}
\begin{rem}
Proposition~13 in \cite{HO11} asserts that a projective limit of strictly real Banach algebras is commutative modulo the radical. Since the spectrum of every element of polynomial growth is contained in $\R$ \cite[Proposition 2.5]{ArOld}, in the real case we can use the cited result instead of Lemma~\ref{PrlmodR}. In fact, the proof in \cite{HO11} is essentially the same as for Lemma~\ref{PrlmodR} but with a few typos.
\end{rem}

\section{Universal constructions}
\label{s:uncon}

In this section, we fix a class $\mathsf{C}$ of (real or complex) Banach algebras stable  under passing to finite products and closed subalgebras.

\subsection*{Tensor and free $\mathsf{CL}$ algebras}

The following notions are introduced in \cite{ArNew}.

\begin{df}\label{Cinften}\cite[Definition~1.6
]{ArNew}
Let $E$ be a complete locally convex space. A \emph{tensor $\mathsf{CL}$ algebra} of $E$ is a $\mathsf{CL}$ algebra $T^{\mathsf{C}}(E)$ equipped with a continuous linear map $\mu\! :E\to T^{\mathsf{C}}(E)$ that is the identity component of the left adjoint to the forgetful functor from $\mathsf{CL}$ to the category of complete locally convex spaces.
\end{df}

In fact, the following definition is a partial case of the previous.

\begin{df}\label{freePGL}\cite[Definition~1.10
]{ArNew}
Let $X$ be a set. A \emph{free $\mathsf{CL}$ algebra} with generating set~$X$ is a $\mathsf{CL}$ algebra  $\cF^{\mathsf{C}}\{X\}$ equipped with a map $\mu\!:X\to \cF^{\mathsf{C}}\{X\}$ that satisfies the following universal property: for any $\mathsf{CL}$ algebra~$B$ and any
map $\psi\!: X \to B$ there is a unique continuous homomorphism
$\widehat\psi\!:\cF^{\mathsf{C}}\{X\}\to B$ such that the diagram
\begin{equation*}
  \xymatrix{
X \ar[r]^{\mu}\ar[rd]_{\psi}&\cF^{\mathsf{C}}\{X\}\ar@{-->}[d]^{\widehat\psi}\\
 &B\\
 }
\end{equation*}
is commutative.
\end{df}

The following three results are proved in \cite{ArNew}.

\begin{pr}\label{exisCiten}
\cite[Propostion~1.8
]{ArNew}
The tensor $\mathsf{CL}$ algebra of a complete locally convex space exists and is unique up to natural isomorphism.
\end{pr}

\begin{pr}\label{exisfree}
\cite[Propostion~1.11
]{ArNew}
The free $\mathsf{CL}$  algebra of a set exists and is unique up to natural isomorphism.
\end{pr}

\begin{pr}\label{PGLisqu}
\cite[Propostion~1.9
]{ArNew}
Every $\mathsf{CL}$ algebra is a quotient of some tensor $\mathsf{CL}$ algebra.
\end{pr}

\subsection*{Symmetric $\mathsf{CL}$  algebras}

\begin{df}\label{Cinfsymm}
Let $E$ be a complete locally convex space. A \emph{symmetric $\mathsf{CL}$  algebra} of $E$ is a commutative $\mathsf{CL}$ algebra $S^{\mathsf{C}}(E)$ equipped with a continuous linear map $\mu\!:E\to S^{\mathsf{C}}(E)$ that is the identity component of the left adjoint to the forgetful functor from the category of commutative  $\mathsf{CL}$ algebras to the category of complete locally convex spaces.
\end{df}

\begin{rem}
Thus the symmetric $\mathsf{CL}$  algebra of $E$ satisfies the following universal property:
for any commutative $\mathsf{CL}$ algebra $B$ and any continuous linear map
$\psi\!: E \to B$ there is a unique continuous homomorphism
$\widehat\psi\!:S^{\mathsf{C}}(E)\to B$ such that the diagram
\begin{equation*}
  \xymatrix{
E \ar[r]^{\mu}\ar[rd]_{\psi}&S^{\mathsf{C}}(E)\ar@{-->}[d]^{\widehat\psi}\\
 &B\\
 }
\end{equation*}
is commutative.
\end{rem}

\begin{pr}\label{exisCisym}
The symmetric $\mathsf{CL}$ algebra  of a complete locally convex space exists and is unique up to natural isomorphism.
\end{pr}

\begin{proof}
The uniqueness follows from the definitions.

The argument for the existence part is similar to that for Proposition~\ref{exisCiten} (see \cite[Propostion~1.8
]{ArNew}) but with applying the envelope in $\mathsf{C}$ to the \emph{Arens--Michael symmetric algebra} of $E$ instead of the Arens--Michael tensor algebra $\wh{T}(E)$. The definition of the Arens--Michael symmetric algebra goes back to A.~Colojoar\u{a}; see \cite[\S\,6.4, p.\,354, Definition 6.54]{Di81}. The existence for every $E$ follows from the fact that the correspondence $A\mapsto (A/\,\overline{[A,A]}\,)\sptilde$ extends to a functor left adjoint to the embedding functor from the category of commutative Arens--Michael  algebras to the category of all Arens--Michael  algebras. (Here $[A,A]$ stands for the ideal generated by all the commutators in an associative algebra $A$.)
\end{proof}

The proof of the following result is similar to Proposition~\ref{PGLisqu} (see \cite[Propostion~1.19%
]{ArNew}) with $T^{\mathsf{C}}(A)$ replaced by $S^{\mathsf{C}}(A) $.

\begin{pr}\label{cPGLisqu}
Every commutative $\mathsf{CL}$ algebra is a quotient of some symmetric $\mathsf{CL}$ algebra.
\end{pr}

\subsection*{Coproducts, colimits and tensor products}

Now we consider coproducts,  which  are also called free products in categories of associative algebras,  and related constructions.

\begin{pr}\label{PGLcopr}
Coproducts exist in $\mathsf{CL}$.
\end{pr}
\begin{proof}
Let $\{A_i\}_{i\in I}$ be a family of $\mathsf{CL}$ algebras. Note that coproducts exist in the category $\mathsf{AM}$ of Arens--Michael algebras; see \cite[\S\,4]{Pi15}. By Proposition~\ref{exisenvgen}, the forgetful functor from $\mathsf{CL}$ to $\mathsf{TA}$ admits a left adjoint functor and so does  the forgetful functor from $\mathsf{CL}$ to $\mathsf{AM}$ since the latter is a subcategory of $\mathsf{TA}$. Therefore, taking first the coproduct of all $A_i$ in $\mathsf{AM}$ and then its envelope with respect to $\mathsf{CL}$, we obtain a coproduct in $\mathsf{CL}$.
\end{proof}

We denote the coproduct of a family $\{A_i\}_{i\in I}$ of $\mathsf{CL}$ algebras  by ${\ast^{\mathsf{C}}_{i\in I}}A_i$ and the coproduct of a finite family $A_1,\ldots,A_n$ by $A_1\mathbin{\ast^{\mathsf{C}}\!}\cdots\mathbin{\ast^{\mathsf{C}}\!}A_n$.

\begin{pr}\label{copoffr}
Let $X$  and $Y$ be sets. Then $\cF^{\mathsf{C}}\{X\}\mathbin{\ast^{\mathsf{C}}\!}\cF^{\mathsf{C}}\{Y\}\cong \cF^{\mathsf{C}}\{X\coprod Y\}$.
\end{pr}
The proof is straightforward.

\begin{pr}\label{PGLcocomp}
Suppose, in addition, that $\mathsf{C}$ is closed under quotients.
Then the category $\mathsf{CL}$ is cocomplete, i.e.,  colimits exist.
\end{pr}
\begin{proof}
It is well known that it suffices to show that coproducts and coequalizers of parallel pairs exist; see, e.g.,  \cite[p.\,60, Theorem 2.8.1]{BorI}. The first property is established in  Proposition~\ref{PGLcopr}. The second property can be proved in a standard way. Indeed, let $\al,\be\!: A\to B$ be homomorphisms in $\mathsf{CL}$. If $I$ is the closed ideal of $B$ generated by elements of the form $\be(a)-\al(a)$, where $a\in A$, then $(B/I)\sptilde$ is in $\mathsf{CL}$  by Proposition~\ref{PGLqu} and
$B\to (B/I)\sptilde$ is a coequalizer.
\end{proof}

The following relative version of coproduct, similar to tensor product of algebras, is also of interest.

\begin{df}\label{PGtp}
Let $\{A_i\}_{i\in I}$ be a family of $\mathsf{CL}$ algebras. The \emph{tensor $\mathsf{CL}$ product} of $\{A_i\}$  is a $\mathsf{CL}$ algebra $A$ equipped with a family $\{\al_i\!:A_i\to A\}$  of continuous homomorphisms with pairwise commuting ranges satisfying the following universal property:
for every $\mathsf{CL}$ algebra $B$ and every family $\{\be_i\!:A_i\to B\}$  of continuous homomorphisms with pairwise commuting ranges there is a unique continuous homomorphism $\psi\!:A\to B$ such that $\be_i=\al_i\psi$ for every $i$.
\end{df}

\begin{pr}\label{exisCtenpr}
Suppose, in addition, that $\mathsf{C}$ is stable under passing to quotients.
Then the tensor $\mathsf{CL}$ product exists for every family in $\mathsf{CL}$ and is isomorphic to the completion of a quotient of the coproduct of this family.
\end{pr}
\begin{proof}
Let $\{A_i\}_{i\in I}$ be a family of $\mathsf{CL}$ algebras. Put $C\!:={\ast^{\mathsf{C}}_{i\in I}}A_i$ and denote by~$I$  the closed ideal in $C$ generated by the elements of the form $[a_i,a_j]$, where $a_i\in A_i$, $a_j\in A_j$  and $i\ne j$. We claim the family $A_i\to C\to (C/I)\sptilde$ of homomorphisms satisfies the required universal property.

Indeed, let $B$ be a $\mathsf{CL}$ algebra  and $\{\be_i\!:A_i\to B\}$  a family  of continuous homomorphisms with pairwise commuting ranges. Take the unique continuous homomorphism $\phi\!:C\to B$ extending all $\be_i$. Then $\psi(I)=0$. Hence there is a unique continuous homomorphism $\psi\!:(C/I)\sptilde \to B$ such that $\phi$ is the composition of the quotient homomorphism and~$\psi$. Since $(C/I)\sptilde$ is in $\mathsf{CL}$ by Proposition~\ref{PGLqugen}, the proof is complete.
\end{proof}

We denote the tensor $\mathsf{CL}$ product of a family $\{A_i\}_{i\in I}$ of $\mathsf{CL}$ algebras by ${\otimes^{\mathsf{C}}_{i\in I}}A_i$ and the tensor $\mathsf{CL}$ product of a finite family $A_1,\ldots,A_n$ by $A_1\mathbin{\otimes^{\mathsf{C}}\!}\cdots\mathbin{\otimes^{\mathsf{C}}\!}A_n$.

\begin{rems}\label{exdeCitp}
(A)~The \emph{Arens--Michael tensor product} ${\otimes^{\,\mathsf{AM}}_{i\in I}}A_i$ of a family $\{A_i\}$ of Arens--Michael algebras is a partial case of Definition~\ref{PGtp} if the class of all Banach algebras is taken. When all $A_i$ are Arens--Michael algebras, it is not hard to see that ${({\otimes^{\,\mathsf{AM}}_{i\in I}}A_i)\sphat}{\,\,}^{\mathsf{C}}\cong {\otimes^{\mathsf{C}}_{i\in I}}\wh{A}_i^{\mathsf{C}}$. This gives another way of proving the existence in Proposition~\ref{exisCtenpr}.

(B)~It is well known that the finite projective tensor product $A_1\ptn \cdots \ptn A_n$ of a tuple $A_1,\ldots, A_n$ of Arens--Michael algebras satisfies the desired universal property; see, e.g., \cite[Chapter 2, \S\,5, pp.\,117--118, Theorem~5.2 and Exercise~5.4]{X1}. Thus it is isomorphic to the Arens--Michael tensor product $A_1\mathbin{\otimes^{\,\mathsf{AM}}\!}\cdots\mathbin{\otimes^{\,\mathsf{AM}}\!}A_n$.
\end{rems}

\begin{rem}\label{colfincop}
Note that infinite coproducts and tensor products can be approximated by finite ones. Indeed, let $\{A_i\}_{i\in I}$ be a family of $\mathsf{CL}$ algebras. For every finite tuple $\la$ of indices  in $I$ put $C_\la\!:={\ast^{\mathsf{C}}_{i\in \la}}A_i$.  Ordering the set of finite tuples by inclusion we obtain a directed system of $\mathsf{CL}$ algebras. A standard categorical argument shows that ${\ast^{\mathsf{C}}_{i\in I}}A_i\cong \varinjlim C_\la$ in $\mathsf{CL}$. Similarly, ${\otimes^{\mathsf{C}}_{i\in I}}A_i\cong \varinjlim D_\la$, where $D_\la\!:={\otimes^{\mathsf{C}}_{i\in \la}}A_i$. Cf.~\cite[Exercise 11B]{AHS}.
\end{rem}

\subsection*{The case when $\mathsf{C}=\mathsf{PG}$}
Note first that the class $\mathsf{PG}$ is stable under passing to closed subalgebras, finite products and
quotients  \cite[Propositions 2.11(A) and~2.3]{ArOld}, so all the results in this section can be applied to it. To emphasize the relationship with the theory of $C^\infty$-differentiable algebras, we also refer to tensor $\mathsf{PGL}$ algebras, free $\mathsf{PGL}$ algebras, symmetric $\mathsf{PGL}$ algebras and tensor $\mathsf{PGL}$ products as \emph{$C^\infty$-tensor algebras}, \emph{$C^\infty$-free algebras}, \emph{$C^\infty$-symmetric algebras} and \emph{$C^\infty$-tensor products}, respectively.

In the next section we need the following result on $C^\infty$-tensor products of $\mathsf{QDTS}$ algebras (see Definition~\ref{QPSDTde}.

\begin{pr}\label{prliCitp2}
If $A_1,\ldots, A_n$ are in $\mathsf{QDTS}$, then $A_1\mathbin{\otimes^{\,\mathsf{PG}}\!}\cdots \mathbin{\otimes^{\,\mathsf{PG}}\!} A_n\cong A_1\ptn\cdots\ptn A_n$.
\end{pr}
\begin{proof}
It suffices to consider the case when $n=2$. By Proposition~\ref{QprliCitp}, $A_1\ptn A_2$ is in  $\mathsf{QDTS}$ and hence by Proposition~\ref{QSDTPGL} in $\mathsf{PGL}$. Since $A_1\ptn A_2$ satisfies the universal property for pairs of continuous homomorphisms with commuting ranges to Arens--Michael algebras, the conditions of Definition~\ref{PGtp} hold for $A_1\ptn A_2$. Thus  $A_1\mathbin{\otimes^{\,\mathsf{PG}}\!}A_2\cong A_1\ptn A_2$.
\end{proof}

We are mainly interested in finitely generated algebras. In the case when $X$ is a finite set of cardinality~$k$, we put $\cF_{k}^{\,\mathsf{PG}}\!:=\cF^{\,\mathsf{PG}}\{X\}$. This object can be called `a finite-rank algebra of free $C^\infty$-functions'; see~\cite{ArNew}. Note that $C^\infty$-tensor algebra $\cF_{k}^{\,\mathsf{PG}}$ is isomorphic to the envelope of the free associative algebra $\R\langle x_1,\ldots,x_k\rangle$  with respect to $\mathsf{PG}$ while
$S^{\,\mathsf{PG}}(E)$, the $C^\infty$-symmetric algebra of an $k$-dimensional space~$E$, is isomorphic
the envelope of  $\R[x_1,\ldots,x_k]$,  which in turn is isomorphic to $C^\infty(\R^k)$; see \cite[Proposition~3.3]{ArNew}.

The the following theorem is implied by the structure results placed in Appendix, Theorems~\ref{multCiffk} and~\ref{fuctcafree}.

\begin{thm}\label{enUfk}
Let $k\in \N$. Then

\emph{(A)}~$\cF_{k}^{\,\mathsf{PG}}$ is a nuclear Fr\'echet space;

\emph{(B)}~$\cF_{k}^{\,\mathsf{PG}}$ is  in $\mathsf{DTS}$.
\end{thm}
\begin{proof}
(A)~By Theorem~\ref{fuctcafree}, the underlying locally convex space of $\cF_{k}^{\,\mathsf{PG}}$ is isomorphic to the space $C^\infty_{\ff_k}$ defined in~\eqref{defCifk}. Note that
$C^\infty_{\ff_k}$  is obtained by applying of the operations of Cartesian product and finite projective tensor product to spaces of the form $C^\infty(M)$, where $M$ is a finite-dimensional real vector space. Since spaces of the form $C^\infty(M)$ are nuclear Fr\'echet spaces, so is $C^\infty_{\ff_k}$.

Part~(B) is contained in Theorem~\ref{multCiffk}.
\end{proof}

\section{Finitely $C^\infty$-generated algebras}
\label{s:finCial}

In this section we introduce our main subject of study --- a finitely $C^\infty$-generated (associative) algebra.  This notion is on the one hand a non-commutative version of the notion of $C^\infty$-differentiable algebra, but on the other hand it is a $C^\infty$ version of the notion of holomorphically finitely generated algebra.

Recall that every $\mathsf{PGL}$ algebra is a quotient of some $C^\infty$-tensor algebra; see Proposition~\ref{PGLisqu}. The similar result holds in the holomorphic case; namely, every  Arens--Michael algebra is a quotient of some analytic tensor algebra. In \cite{Pi14, Pi15} Pirkovskii introduced holomorphically finitely generated algebras as quotients of analytic tensor algebras of finite-dimensional spaces. We now consider a  similar subclass in $\mathsf{PGL}$. Also, every $C^\infty$-differentiable algebra (i.e., a quotient of some $C^\infty(\R^k)$) is a $\mathsf{PGL}$ algebra; see Corollary~\ref{CdiffPGL}. Here we replace the free commutative  algebra $C^\infty(\R^k)$ by the free $\mathsf{PGL}$ algebra $\cF_{k}^{\,\mathsf{PG}}$ (eq., the $C^\infty$-tensor  algebra with $k$-dimensional generating space).

\begin{df}
A real Arens--Michael algebra is said to be \emph{finitely $C^\infty$-generated} if it is isomorphic to a quotient of~$\cF_{k}^{\,\mathsf{PG}}$  for some $k\in\N$.
\end{df}
We have two reasons for introducing this class: our main examples are  finitely $C^\infty$-generated and the good behaviour of this class with respect to the projective tensor product.

\begin{pr}
Every finitely $C^\infty$-generated algebra is a nuclear Fr\'echet space.
\end{pr}
\begin{proof}
Let $k\in \N$.
By Theorem~\ref{enUfk}, $\cF_{k}^{\,\mathsf{PG}}$ is a nuclear Fr\'echet space and so are all its quotients.
\end{proof}

Finitely $C^\infty$-generated algebras form a full subcategory in $\mathsf{PGL}$. We also consider the smaller subcategory of commutative finitely $C^\infty$-generated algebras.

\begin{pr}\label{cifgcial}
A commutative real Arens--Michael algebra is finitely $C^\infty$-generated if and only if it is a $C^\infty$-differentiable algebra.
\end{pr}
\begin{proof}
Take $k\in\N$ and note that $C^\infty(\R^k)$ is the quotient of $\cF_{k}^{\,\mathsf{PG}}$ by the closed ideal generated by commutators. So every $C^\infty$-differentiable algebra, being a quotient of some $C^\infty(\R^k)$, is also a quotient of $\cF_{k}^{\,\mathsf{PG}}$.

On the other hand, every commutative quotient of $\cF_{k}^{\,\mathsf{PG}}$ is a quotient of $C^\infty(\R^k)$ because the quotient map factors on $\cF_{k}^{\,\mathsf{PG}}\to C^\infty(\R^k)$.
\end{proof}

Recall that the \emph{real spectrum} of an $C^\infty$-differentiable algebra  $A$ is the locally ringed $\R$-space $(\mathop{\mathrm{Spec}}_r A,\, \wt A)$, where $\mathop{\mathrm{Spec}}_r A$ is the set of homomorphisms to $A$ to $\R$ and $\wt A$ is a naturally defined sheaf of $C^\infty$-differentiable algebra;  for details see \cite[\S\,3]{NaSa}.
Also, a locally ringed $\R$-space $(X,\,\cO_X)$ is said to be an \emph{affine $C^\infty$-differentiable space} (of finite type) if it is isomorphic to the real spectrum of some $C^\infty$-differentiable algebra \cite[p.\,44]{NaSa}.

The functors $(X,\,\cO_X)\rightsquigarrow\cO_X(X)$ and $A \rightsquigarrow (\mathop{\mathrm{Spec}}_r A,\, \wt A)$  define an anti-equivalence of the category of affine $C^\infty$-differentiable spaces with the category of $C^\infty$-differentiable algebras \cite[p.\,47, Theorem 3.20]{NaSa}. In view of the fact, that every homomorphism of $C^\infty$-differentiable algebras is automatically continuous \cite[p.\,33, Corollary 2.22]{NaSa}, we obtain the following corollary of Proposition~\ref{cifgcial}.

\begin{co}\label{eqcat}
The functors $(X,\,\cO_X)\rightsquigarrow\cO_X(X)$ and $A \rightsquigarrow (\mathop{\mathrm{Spec}}_r A,\, \wt A)$ define an anti-equivalence of the category of affine $C^\infty$-differentiable spaces with the category of commutative finitely $C^\infty$-generated  algebras.
\end{co}

This result is an analogue of the Forster--Pirkovskii theorem, which states that the global sections functor is an anti-equivalence between the category of Stein spaces of finite embedding dimension and the category of commutative holomorphically finitely generated algebras  \cite[Theorem  3.23]{Pi15}.

\begin{pr}\label{fcgBrad}
A finitely $C^\infty$-generated algebra $A$ is commutative modulo radical and  $A/\Rad A$ is a $C^\infty$-differentiable algebra.
\end{pr}
\begin{proof}
Since $A$ is in $\mathsf{PGL}$, the first assertion is implied by Proposition~\ref{PGLcmR}.
The second assertion follows from Proposition~\ref{cifgcial}.
\end{proof}

Now we turn to envelopes.

\begin{pr}\label{envfing}
\emph{(A)}~If $A$ is a finitely generated real algebra, then $\wh A^{\,\mathsf{PG}}$ is finitely $C^\infty$-generated.

\emph{(B)}~If, in addition, $A$ is commutative, then $\wh A^{\,\mathsf{PG}}$ is a $C^\infty$-differentiable algebra.
\end{pr}
\begin{proof}
(A)~Let $A=\R\langle x_1,\ldots,x_k\rangle/I$ for some $k$ and some ideal $I$.
Since $\cF_{k}^{\,\mathsf{PG}}$ is a Fr\'echet space, it follows from  Proposition~\ref{PGquot} in the case when $\mathsf{C}=\mathsf{PG}$ that $\wh A^{\,\mathsf{PG}}\cong \cF_{k}^{\,\mathsf{PG}}/J$, where $J$ is a closed ideal. Thus $\wh A^{\,\mathsf{PG}}$ is finitely $C^\infty$-generated.

Part~(B) follows from Part~(A) and Proposition~\ref{fcgBrad}.
\end{proof}

\begin{rem}\label{qalg}
Denote  by $\mathcal{R}(\R^2_q)$ the coordinate algebra of the quantum plane over~$\R$, i.e., the universal real algebra with generators $x$ and $y$ satisfying the relation $xy = qyx$, where $q\in\R\setminus\{0 \}$.

For $q\in\R\setminus\{0\}$ also denote  by $\cR(\SL_q(2,\R))$ the coordinate algebra of the quantum group $\SL_q(2,\R)$, i.e.,  the universal real associative algebra with generators $a$, $b$, $c$, $d$ and relations $ab = qba$, $ac = qca$, $bc = cb$,
$bd = qdb$, $cd = qdc$, $da- q^{-1}bc=1$ and $ad - qbc=1$;
see, e.g., \cite[\S\,4.1.2]{KSc}.

By Proposition~\ref{envfing}, the envelopes  of the algebras $\mathcal{R}(\R^2_q)$ and $\cR(\SL_q(2,\R))$ with respect to $\mathsf{PG}$ are finitely $C^\infty$-generated and, moreover, so is  the envelope of $U(\fg)$ for a finite-dimensional Lie algebra~$\fg$. Explicit forms of these envelopes are described in~\cite[\S\,3]{ArNew}. For a discussion on Hopf algebra structures see \S\,\ref{s:Hopf}.
\end{rem}

\begin{co}\label{finfCi}
Every finite-dimensional algebra of polynomial growth is finitely $C^\infty$-generated.
\end{co}
\begin{proof}
If $A$ is a finite-dimensional algebra of polynomial growth, then it is a
Banach algebra  of polynomial growth and then $\wh A^{\,\mathsf{PG}}=A$.
Since a finite-dimensional algebra is finitely generated, we can apply Proposition~\ref{envfing}.
\end{proof}

The following two results is related with the discussion on Hopf algebras in \S\,\ref{s:Hopf}.

\begin{pr}\label{tensfing0}
If $A_1,\ldots, A_n$ are finitely $C^\infty$-generated algebras, then
$$
A_1\mathbin{\otimes^{\,\mathsf{PG}}\!}\cdots \mathbin{\otimes^{\,\mathsf{PG}}\!} A_n\cong A_1\ptn\cdots\ptn A_n.
$$
\end{pr}
\begin{proof}
Note that every finitely $C^\infty$-generated algebra is in $\mathsf{QDTS}$ since it is a quotient of an algebra of the form $\cF_{k}^{\,\mathsf{PG}}$, which is in $\mathsf{DTS}$ by Theorem~\ref{enUfk}. So we can apply Proposition~\ref{prliCitp2}.
\end{proof}

\begin{thm}\label{tensfing}
If $A_1,\ldots, A_n$ are finitely $C^\infty$-generated algebras, then so is $A_1\ptn\cdots\ptn A_n$.
\end{thm}
\begin{proof}
It suffices to consider the case when $n=2$. By the definition, $A_1$ and $A_2$ are quotients of $\cF_{k_1}^{\,\mathsf{PG}}$ and $\cF_{k_2}^{\,\mathsf{PG}}$, respectively, for some $k_1,k_2\in\N$.  Since we are dealing with Fr\'echet algebras, a standard property of projective tensor products implies that $A_1\ptn A_2$ is a quotient of $\cF_{k_1}^{\,\mathsf{PG}}\ptn \cF_{k_2}^{\,\mathsf{PG}}$, which is isomorphic to $\cF_{k_1}^{\,\mathsf{PG}}\mathbin{\otimes^{\,\mathsf{PG}}\!}\cF_{k_2}^{\,\mathsf{PG}}$ by Proposition~\ref{tensfing0}. Moreover, it follows from Proposition~\ref{exisCtenpr} that $\cF_{k_1}^{\,\mathsf{PG}}\mathbin{\otimes^{\,\mathsf{PG}}\!}\cF_{k_2}^{\,\mathsf{PG}}$ is a quotient of $\cF_{k_1}^{\,\mathsf{PG}}{\ast^{\mathsf{PG}}}\cF_{k_2}^{\,\mathsf{PG}}$, which in turn is isomorphic to $\cF_{k_1+k_2}^{\,\mathsf{PG}}$ by Proposition~\ref{copoffr}.
Thus $A_1\ptn A_2$ is a quotient of the last algebra and so is finitely $C^\infty$-generated.
\end{proof}

\section{Example of a finitely $C^\infty$-generated algebra}
\label{s:af1}

In this section we construct a finitely $C^\infty$-generated algebra that is not  obtained as an envelope of a finitely generated algebra, but directly as a quotient of a $C^\infty$-free algebra. This example is a deformation of the envelope of $U(\fa\ff_1)$, where $\fa\ff_1$ is the non-abelian solvable $2$-dimensional Lie algebra. The first form of this  deformation was introduced in \cite{AiSa} at the physical level of rigour; for a holomorphic version see \cite{AHHFG}.

Recall that $\fa\ff_1$ is generated by elements $e_1,e_2$ subject to the relation $[e_1,e_2]=e_2$.
Now let $\hbar\in\R\setminus \{0\}$ and denote by $(C^\infty_{\fa\ff_1})_\hbar$ the universal finitely $C^\infty$-generated algebra with generators $e_1,e_2$ subject to the relation
\begin{equation}\label{af1qu}
[e_1,e_2]=\frac{\sinh (\hbar e_2)}{\sinh \hbar}.
\end{equation}
This means that $(C^\infty_{\fa\ff_1})_\hbar$ is the quotient of $\cF_{2}^{\,\mathsf{PG}}$ over the closed ideal generated by the element $[e_1,e_2]\sinh \hbar-\sinh (\hbar e_2)$, which is well defined because each element of a $\mathsf{PGL}$ algebra admits a $C^\infty$-functional calculus; see, e.g., \cite[Theorem 3.2]{ArOld}.

Here we describe the structure of $\mathsf{PGL}$ algebra  on $(C^\infty_{\fa\ff_1})_\hbar$. For the Hopf $\ptn$-algebra structure see Proposition~\ref{af1qco} in the next section.  Of course, if we were not interested in such a structure, we could replace the function on the right-hand side by another one that has similar properties in the neighborhood of zero.

\begin{thm}\label{multCifgqua}
Let $\hbar\in\R\setminus \{0\}$.

\emph{(A)}~The Fr\'echet space $C^\infty(\R)\ptn \R[[e_2]]$ is a $\mathsf{DTS}$ algebra with respect to a multiplication such that~\eqref{af1qu} holds.

\emph{(B)}~With respect to this multiplication, $C^\infty(\R)\ptn \R[[e_2]]$ is a universal finitely $C^\infty$-generated  algebra, i.e., it is topologically isomorphic to $(C^\infty_{\fa\ff_1})_\hbar$.
\end{thm}

Thus the underlying Fr\'echet space of $(C^\infty_{\fa\ff_1})_\hbar$ is the same as in the undeformed case of~$C^\infty_{\fa\ff_1}$; see \cite[Example 4.7]{ArOld}. But, of course, the multiplication is different.

We need auxiliary results. The following theorem is a strengthening of a well-known result on ordered $C^\infty$-functional calculus; see a discussion in \cite{ArOld}.

\begin{thm}
\label{fucani}
\cite[Theorem~3.3]{ArOld}
Let $b_1, \dots, b_m$ be elements of an Arens--Michael $\mathbb{R}$-algebra~$B$. Suppose that $b_1,\dots,b_k$ \emph{(}$k\le m$\emph{)} are of polynomial growth and $b_{k+1},\dots,b_m$ are nilpotent. Then the linear map $\mathbb{R}[\lambda_1,\dots,\lambda_m]\to B$ taking the (commutative) monomial $\lambda_1^{\beta_1}\cdots \lambda_m^{\beta_m}$ to the (non-commutative) monomial $b_1^{\beta_1}\cdots b_m^{\beta_m}$ extends to a continuous linear map
$$
C^\infty(\mathbb{R}^k)\mathbin{\widehat{\otimes}} \mathbb{R}[[\lambda_{k+1},\dots,\lambda_m]]\to B.
$$
\end{thm}

Now we use this theorem in our situation.

\begin{lm}\label{af1dext}
If $x$ and $z$ are elements of a $\mathsf{PGL}$ algebra $B$ such that $[x,z]=\sinh (\hbar z)/\sinh \hbar$,
then the linear map $\R[e_1,e_2]\to B\!:e_1\mapsto x,\,e_2\mapsto z$  extends to a continuous linear map defined on $C^\infty(\R)\ptn \R[[e_2]]$.
\end{lm}
\begin{proof}
We can assume that $B$ is a Banach algebra  of polynomial growth. The relation in the hypothesis implies that $\sinh (\hbar z)$ belongs to the commutant and therefore belongs to $\Rad B$ by Proposition~\ref{PGLcmR}. On the other hand, $\Rad B$ is nilpotent by \cite[Theorem~2.9]{ArOld}. Thus $\sinh (\hbar z)$ is nilpotent.

We claim that $z$ is also nilpotent. Indeed, it follows from the composition property of holomorphic functional calculus \cite[p.\,218, Theorem 4.95]{AlDa} that $a=\arcsinh (\sinh a)$ when $a$ is an element of a Banach algebra such that the spectrum of $\sinh a$ is sufficiently small. In particular, this holds when $\sinh a$ is nilpotent. Hence $z=(\arcsinh y)/\hbar$, where $y=\sinh (\hbar z)$. Since $y$ is nilpotent, $z$ is a polynomial in $y$. Moreover, the free term of this polynomial is~$0$ and thus $z$ is also nilpotent. The claim is proved.

Since $x$ is of polynomial growth and $z$ is nilpotent, it follows from Theorem~\ref{fucani} that there is a continuous linear map $C^\infty(\R)\ptn \R[[e_2]]\to B$ taking $e_1$ and $e_2$ to $x$ and $z$, respectively.
\end{proof}

\begin{proof}[Proof of Theorem~\ref{multCifgqua}]
(A)~We construct a sequence  $(\pi_p)_{p\in\N}$ of triangular representations of~\eqref{af1qu}. Let
\begin{equation*}
X_p:= \begin{pmatrix}
p& &&& \\
 &p-1 && &\\
 & & \ddots&&\\
 & & &1&\\
 & & &&0
\end{pmatrix},\qquad
E_p:= \begin{pmatrix}
0& 1&&& \\
&0& 1&& \\
 && \ddots& \ddots&\\
 & &&0 &1\\
 & &&&0
\end{pmatrix}
\end{equation*}
be matrices in $\rT_{p+1}$.

To find matrices satisfying~\eqref{af1qu} we first consider a simplified relation and assume that
\begin{equation}\label{XPsinh}
 [X_p,Z_p]=\frac{\sinh (\hbar Z_p)}{\hbar},
\end{equation}
where
$Z_p=E_p+\al_2 E_p^2+\cdots+\al_p E_p^p$
for some  $\al_2,\ldots,\al_p\in\R$. Since $[X_p,E_p^j]=jE_p^j$ for every $j$, Equality~\eqref{XPsinh} takes the form
$$
E_p+2\al_2 E_p^2+\cdots+p\al_p E_p^p=(E_p+\al_2 E_p^2+\cdots+\al_p E_p^p)+\frac{\hbar^2}{3!}(E_p+\al_2 E_p^2+\cdots+\al_p E_p^p)^3+\cdots.
$$
This relation gives a  system of linear equations with $\al_2,\ldots,\al_p$ as unknowns, which is finite because $E^{p+1}=0$. Note that $\al_j=0$ for even~$j$ since Taylor's expression for the hyperbolic sine contains only odd degrees. For odd $j$ we have the system
$$
\begin{cases}
3\al_3=\al_3+\frac{\hbar^2}{3!}\\
5\al_5=\al_5+\frac{\hbar^2}{2!}\al_3+\frac{\hbar^4}{5!}\\
\cdots
\end{cases}.
$$
It is not hard to see that all $\al_j\ne 0$ in this case. Also, they are obviously independent of~$p$.

Now put $\ka\!:=\hbar/\sinh \hbar$ and define $\pi_p(e_1)=\ka X_p$ and $\pi_p(e_2)=Z_p$. Since~\eqref{XPsinh} holds, we have that $[\ka X_p,Z_p]=\sinh \hbar Z_p/\sinh \hbar$. Thus $\pi_p$ is the desired representation.

For small values of $p$ we get
$$
 \pi_0(e_1)=0, \qquad\pi_0(e_2)=0;
$$
$$
\pi_1(e_1)=\ka
\begin{pmatrix}
 1& 0\\
   & 0
\end{pmatrix},
\qquad
\pi_1(e_2)=\begin{pmatrix}
 0 &1\\
    &0
\end{pmatrix};
$$
$$
\pi_2(e_1)=\ka\begin{pmatrix}
 2 &0 &0\\
  & 1& 0\\
   & & 0
\end{pmatrix},
\qquad
\pi_2(e_2)=\begin{pmatrix}
 0 & 1&0\\
 & 0 &1\\
   &  &0
\end{pmatrix};
$$
$$
\pi_3(e_1)=\ka
\begin{pmatrix}
3&0 &0&0 \\
 &2 &0 &0\\
 & &1 &0\\
   & & &0
\end{pmatrix},
\qquad
\pi_3(e_2)=\begin{pmatrix}
0&1 &0&\frac{\hbar^2}{12}\\
&0 & 1&0\\
 & & 0 &1\\
   & & &0
\end{pmatrix}.
$$
Thus for $p\le 2$ the representation $\pi_p$ has the same form (up to scalar) as in the undeformed case in \cite[Example~4.7]{ArOld} but it is different for $p>2$.

For every $\la\in\R$ the formulas $\pi_{p,\la}(e_1)=\ka X_p+\la$ and $\pi_{p,\la}(e_2)=Z_p$ also define elements of $\rT_{p+1}$ satisfying~\eqref{af1qu}. Treating $\la$ as a variable, we get elements of $\rT_{p+1}(C^\infty(\R))$, which we denote by $\wt\pi_p(e_1)$ and $\wt\pi_p(e_2)$. By Lemma~\ref{af1dext}, there is a continuous linear map from $\wt\pi_p\!:C^\infty(\R)\ptn \R[[e_2]]\to \rT_{p+1}(C^\infty(\R))$ determined by these elements.

We identify $C^\infty(\R^\times)\ptn \R[[e_2]]$ with $C^\infty(\R)^{\Z_+}$ and take the map
$$
\rho\!:C^\infty(\R)^{\Z_+}\,\rightarrow\, \prod_p \rT_p(C^\infty(\R))
$$
induced by the sequence  $(\wt\pi_p)$. To complete the proof of Part~(A)  it suffices to show that $\rho$ is topologically injective.

Let $a=(f_j)\in C^\infty(\R)^{\Z_+}$ or, in the power series form, $a=\sum_j f_j(e_1)e_2^j$.
Then 
$$
\wt\pi_0(a)=f_0(\la);\qquad
\wt\pi_1(a)=\begin{pmatrix}
f_0(\la+\ka)&f_1(\la+\ka)  \\
 & f_0(\la)
\end{pmatrix};
$$
$$
\wt\pi_2(a)=\begin{pmatrix}
f_0(\la+2\ka)&f_1(\la+2\ka) &f_2(\la+2\ka) \\
 &f_0(\la+\ka) & f_1(\la+\ka)\\
 & & f_0(\la)
\end{pmatrix};
$$
$$
\wt\pi_3(a)=\begin{pmatrix}
f_0(\la+3\ka)&f_1(\la+3\ka) &f_2(\la+3\ka)&\frac{\hbar^2}{12}\,f_1(\la+3\ka)+f_3(\la+3\ka) \\
 &f_0(\la+2\ka) & f_1(\la+2\ka)&f_2(\la+2\ka)\\
 & & f_0(\la+\ka) &f_1(\la+\ka)\\
   & & &f_0(\la)
\end{pmatrix}
$$
etc. Since $E_p$ is nilpotent,  the functions $f_q$ with $q>p$ do not affect on the upper right entry in $\wt\pi_p(a)$. So this entry  
is a linear combination of shifts of $f_0,\ldots,f_p$. Denote by $\rho'$ the continuous linear map $C^\infty(\R)^{\Z_+}\to C^\infty(\R)^{\Z_+}$, where the product of the upper right entries of $\wt\pi_p$ is taken on the right-hand side. It suffices to show that $\rho'$ is topologically injective; see, e.g., Part~(B) of \cite[Lemma 3.10
]{ArNew}.

Write $\rho'$ as the composition
$$
C^\infty(\R)^{\Z_+}\,\xrightarrow{\tau}\, \prod_{p\in\Z_+} C^\infty(\R)^{p+1}\, \xrightarrow{\rho''}\, C^\infty(\R)^{\Z_+},
$$
where $\tau$ takes a sequence to the set of truncations and $\rho''$ is the product of the maps $\rho''_p\!:C^\infty(\R)^{p+1}\to C^\infty(\R)$ taking $(g_0,\ldots, g_p)$ to the upper right entry of $\wt\pi_p(\sum g_j(e_1)e_2^j)$. Since $\tau$ is topologically injective, it suffices to show that $\rho''$ is topologically injective.

Note that each of $\rho''_p$ is surjective. In particular, the range of $\rho''_p$ is closed for every~$p$ and hence  the range of $\rho''$ is closed; see, e.g.,  Part~(A) of \cite[Lemma 3.10
]{ArNew}.

We now claim that $\rho''$ is injective. Suppose that $\rho''_p(g_0,\ldots,g_p)=0$ for every $p$. It is easy to see that each of $\rho''_p(g_0,\ldots,g_p)$ is a linear combination of shifts of $g_0,\ldots,g_p$ with the $p$th coefficient equal to~$1$. This implies immediately that $g_0=0$. Arguing by induction, we conclude that all $g_p$ are equal to~$0$.
Thus $\rho''$ is injective. 

It follows from the inverse mapping theorem for Fr\'echet space that, being injective with closed range, $\rho''$ is topologically  injective. The proof of  Part~(A) is complete.

Part~(B) follows from Part~(A) and Lemma~\ref{af1dext}.
\end{proof}

\section{Hopf algebras}
\label{s:Hopf}

General considerations leading to the following definition are given in Introduction.

\begin{df}
A real Hopf $\mathbin{\widehat{\otimes}}$-algebra that is a finitely $C^\infty$-generated algebra is called a \emph{$C^\infty$-finitely generated Hopf algebra}.
\end{df}

\subsection*{Enveloping functor}

The following proposition provides a sufficient condition for the envelope with respect to $\mathsf{PG}$ to preserve the Hopf $\ptn$-algebra structure.

\begin{pr}\label{PGeHopf}
Let $H$ be a real Hopf $\ptn$-algebra. If the natural homomorphism
$$
\wh H^{\,\mathsf{PG}}\ptn \wh H^{\,\mathsf{PG}}\to \wh H^{\,\mathsf{PG}}\mathbin{\otimes^{\,\mathsf{PG}}\!}\wh H^{\,\mathsf{PG}}
$$
is a topological isomorphism, then there is a unique real Hopf $\ptn$-algebra structure on $\wh H^{\,\mathsf{PG}}$ such that $H\to \wh H^{\,\mathsf{PG}}$ becomes a homomorphism of Hopf $\ptn$-algebras.
\end{pr}
\begin{proof}
Note that $\wh H^{\,\mathsf{PG}}\mathbin{\otimes^{\,\mathsf{PG}}\!}\wh H^{\,\mathsf{PG}}={(\wh H^{\,\mathsf{PG}}\ptn \wh H^{\,\mathsf{PG}})\sphat}{\,\,}^{\,\mathsf{PG}}$; cf. Remarks~\ref{exdeCitp}. So we can use the same argument as in \cite[Proposition 6.7]{Pir_stbflat} applying this formula instead of \cite[Proposition 6.4]{Pir_stbflat}. (It is essential here that the comutiplication $\wh H^{\,\mathsf{PG}}\to\wh H^{\,\mathsf{PG}}\ptn \wh H^{\,\mathsf{PG}}$ is obtained by an application of the universal property since the latter algebra is in $\mathsf{PGL}$.)
\end{proof}

The following result shows that under a finiteness condition the envelope is in fact a functor between categories of Hopf algebras.

\begin{thm}\label{PGenvH}
Let $H$ be a real Hopf algebra that is affine (i.e., finitely generated as an associative algebra).

\emph{(A)}~Then there is a unique real Hopf $\ptn$-algebra structure on $\wh H^{\,\mathsf{PG}}$ such that $H\to \wh H^{\,\mathsf{PG}}$ becomes a homomorphism of Hopf $\ptn$-algebras.

\emph{(B)}~The correspondence $H\mapsto \wh H^{\,\mathsf{PG}}$ extends to a functor from the category of affine real Hopf algebras to the category of $C^\infty$-finitely generated Hopf algebras.
\end{thm}
\begin{proof}
(A)~Proposition~\ref{envfing} implies that $\wh H^{\,\mathsf{PG}}$ is finitely $C^\infty$-generated. So $\wh H^{\,\mathsf{PG}}\ptn \wh H^{\,\mathsf{PG}}\to \wh H^{\,\mathsf{PG}}\mathbin{\otimes^{\,\mathsf{PG}}\!}\wh H^{\,\mathsf{PG}}$ is a topological isomorphism by Proposition~\ref{tensfing0}.

On the other hand, being finitely generated, $H$ has a countable basis.
Every real associative algebra with countable basis is a locally convex
algebra with respect to the strongest locally convex topology. (The proof is the same as for a complex associative algebra; see, e.g., \cite[Proposition 2.1]{Cu05}.) Since the topology is strongest, $H\ptn H\cong H\otimes H$ as a vector space and hence $H$ is a Hopf $\ptn$-algebra. So we can apply Proposition~\ref{PGeHopf}, which completes the proof.

The argument for Part~(B) is the same as for \cite[Proposition 6.7]{Pir_stbflat}.
\end{proof}

\subsection*{Examples of $C^\infty$-finitely generated Hopf algebras}

First, we consider  examples from~\S\,4 
in \cite{ArNew}. Theorem~\ref{PGenvH} immediately implies the following result.

\begin{pr}\label{fgeLie}
Let $\fg$ be a finitely generated real Lie algebra.
Then $\wh U(\fg)^{\,\mathsf{PG}}$ is a $C^\infty$-finitely generated Hopf algebra.
\end{pr}

In addition to the notation $C^\infty_{\ff_k}$, we use a similar notation for a triangular finite-dimensional real Lie algebra~$\fg$. Fix  a linear basis $e_{k+1},\ldots, e_m$ in~$[\fg,\fg]$ and its complement $e_1,\ldots, e_k$ up to a linear basis in~$\fg$. We use the following notation introduced in \cite{ArOld}:
\begin{equation}\label{Cinffgdef}
C^\infty_\fg\!:=C^\infty(\R^k)\ptn \R[[e_{k+1},\ldots,e_m]].
\end{equation}
According to \cite[Theorem~4.12]{ArNew}, $C^\infty_\fg\cong\wh U(\fg)^{\,\mathsf{PG}}$ in the triangular case. Thus  we get the following corollary.

\begin{co}
Let $\fg$ be a triangular finite-dimensional real Lie algebra.
Then $C^\infty_\fg$ is a $C^\infty$-finitely generated Hopf algebra with respect to the operations continuously extended from $U(\fg)$.
\end{co}

Recall that $C^\infty_{\ff_k}\cong\cF_{k}^{\,\mathsf{PG}}\cong \wh U(\ff_k)^{\,\mathsf{PG}}$ as a topological algebra; see Theorem~\ref{fuctcafree}.
Using Theorems~\ref{multCiffk} and \ref{fuctcafree}  we get another corollary of Proposition~\ref{fgeLie}.

\begin{co}
Let $k\in \N$. Then $C^\infty_{\ff_k}$ is a $C^\infty$-finitely generated Hopf algebra with respect to the operations continuously extended from $U(\ff_k)$.
\end{co}

Turning to the quantum groups of the form $\SL_q(2,\R)$ (see Remark~\ref{qalg}), note first that since $\cR(\SL_q(2,\R))$ is finitely generated, Theorem~\ref{PGenvH} implies that ${\cR(\SL_q(2,\R))\sphat}{\,\,}^{\,\mathsf{PG}}$ is a $C^\infty$-finitely generated Hopf algebra.
Following \cite[\S\,4]{ArNew} we denote the latter algebra by $C^\infty(\SL_q(2,\R))$. Note that $C^\infty(\SL_q(2,\R))$ coincides with the envelope of the universal algebra $A_q$ whose generators satisfy the same relations as $\cR(\SL_q(2,\R))$ and the additional assumption that $a$ is invertible; see \cite[Corollary~3.14
]{ArNew}. Note that we may omit $d$ in the set of generators because  $d=a^{-1}(1+ qbc)$ in this case. The following result gives explicit formulas for the operations on $\wh A_q^{\,\mathsf{PG}}$.

\begin{pr}\label{AqHopf}
Let $q\in\R\setminus\{0,1,-1\}$. Then $\wh A_q^{\,\mathsf{PG}}$ is a $C^\infty$-finitely generated Hopf algebra with respect to the comultiplication $\De$, counit $\varepsilon$ and antipode $S$ determined by
\begin{alignat*}{3}
\De\!&:
 a\mapsto a\otimes a+ b\otimes c,&\quad
 b&\mapsto a\otimes b+ b\otimes a^{-1}(1+ qbc),&\quad
 c&\mapsto c\otimes a+ a^{-1}(1+ qbc)\otimes c\,;\\
S\!&:
 a\mapsto a^{-1}(1+ qbc),&\quad
 b&\mapsto-q^{-1}b,&\quad
 c &\mapsto-qc\,;\\
\varepsilon\!&:
 a\mapsto1,&\quad b&,c\mapsto0.&&
\end{alignat*}
\end{pr}
\begin{proof}
As noted above, $C^\infty(\SL_q(2,\R))$ is a $C^\infty$-finitely generated Hopf algebra by Theorem~\ref{PGenvH} and isomorphic to $\wh A_q^{\,\mathsf{PG}}$ by \cite[Corollary~3.14
]{ArNew}. To complete the proof we need to establish the formulas for the operations.

Recall that the operations on $\cR(\SL_q(2,\R))$ is given by the following rules:
\begin{alignat*}{4}
\De\!&:
 a\mapsto a\otimes a+ b\otimes c,&\quad
 b&\mapsto a\otimes b+ b\otimes d,&\quad
 c&\mapsto c\otimes a+ d\otimes c,&\quad
 d&\mapsto c\otimes b+ d\otimes d;\\
S\!&:
 a\mapsto d,&\quad
 b&\mapsto-q^{-1}b,&\quad
 c &\mapsto-qc,&\quad
 d&\mapsto a;\\
\varepsilon\!&:
a,d\mapsto1,&\quad
b&,c\mapsto0;&&&
\end{alignat*}
see, e.g., \cite[\S\S\,4.1.1--4.1.2]{KSc}. By Theorem~\ref{PGenvH}, all the operations extend to  $C^\infty(\SL_q(2,\R))$. It follows from \cite[Lemma~3.12
]{ArNew} that $a$ is invertible and so $d=a^{-1}(1+ qbc)$. Thus,  replacing $d$  with $a^{-1}(1+ qbc)$ we have the  formulas we want.
\end{proof}

\begin{rem}
In fact, $A_q$ is a localization of $\cR(\SL_q(2,\R))$ inverting $a$. But we cannot endow $A_q$ with a Hopf algebra structure given by the formulas in Proposition~\ref{AqHopf}. This follows from that fact that $a\otimes a+ b\otimes c$ is not  invertible in $A_q$ and so $\De$ is not well defined.

On the other hand, $a\otimes a+ b\otimes c$ is invertible in $\wh A_q^{\,\mathsf{PG}}$. This is implied by Proposition~\ref{AqHopf} but can also be shown directly.  It suffices to show this for elements $a,b,c$ of a Banach algebra of polynomial growth such that the first three relations in the definition of $\cR(\SL_q(2,\R))$  hold and $a$ is invertible.
Indeed, since $a\otimes a+ b\otimes c=(a\otimes a)(1\otimes 1+ a^{-1}b\otimes a^{-1}c)$, we need to show that $1\otimes 1+ a^{-1}b\otimes a^{-1}c$ is invertible. Note that $ab=q(q-1)^{-1}[a,b]$ is nilpotent and so is $a^{-1}b$, as well as $a^{-1}b\otimes a^{-1}c$.  Then
$\sum (-1)^n (a^{-1}b)^n\otimes (a^{-1}c)^n$ converges and hence the limit is the inverse element of $1\otimes 1+ a^{-1}b\otimes a^{-1}c$.
\end{rem}

Finally, we consider the algebra $(C^\infty_{\fa\ff_1})_\hbar$ in Theorem~\ref{multCifgqua}; cf. the holomorphic case in \cite[Proposition~5.6]{AHHFG}.

\begin{pr}\label{af1qco}
Let $\hbar\in\mathbb{C}$ and $\sinh\hbar\ne 0$. Then $(C^\infty_{\fa\ff_1})_\hbar$ is a $C^\infty$-finitely generated Hopf algebra with respect to the comultiplication $\De$,  counit $\varepsilon$ and
antipode $S$ determined by
\begin{alignat*}{2}
\De\!&:e_1\mapsto e_1\otimes e^{-\hbar e_2}+ e^{\hbar e_2}\otimes e_1,&\qquad &e_2\mapsto 1\otimes e_2+ e_2\otimes 1;\\
 S\!&:e_1\mapsto -e_1-\frac{\hbar }{\sinh \hbar}\,\sinh (\hbar e_2),& & e_2\mapsto -e_2;\\
 \varepsilon\!&:e_1,e_2\mapsto 0\,,&&
\end{alignat*}
respectively, Moreover, $S$ is invertible.
\end{pr}
\begin{proof}
The argument is the same as for \cite[Proposition 5.6]{AHHFG}, where the case of the universal complex Arens--Michael algebra determined by the same relation is considered. The only difference is that we use the fact that $(C^\infty_{\fa\ff_1})_\hbar$ is a finitely $C^\infty$-generated algebra universal in the class $\mathsf{PGL}$, which implies that $(C^\infty_{\fa\ff_1})_\hbar\mathbin{\otimes^{\,\mathsf{PG}}\!} (C^\infty_{\fa\ff_1})_\hbar\cong (C^\infty_{\fa\ff_1})_\hbar\ptn (C^\infty_{\fa\ff_1})_\hbar$ by Proposition~\ref{tensfing0}.
\end{proof}

Finally, we consider a non-trivial finite-dimensional example.

\begin{exm}
Consider the Taft algebras $H_{n,q}$. Note that over $\R$ the definition of these algebras makes sense only when $n\le 2$. We take $H_{2,-1}$, which is the $4$-dimensional Hopf algebra generated as an associative algebra  by $a$ and $x$ subject to the relations
$a^2 = 1$, $x^2 = 0$ and $xa = -ax$; see, e.g., \cite[\S\,7.3]{Rad}. Being finite dimensional, $H_{2,-1}$ is a Hopf $\ptn$-algebra.

For given $\la\in\R^\times$ consider the following representation of $H_{2,-1}$:
$$
\pi_\la(a)=
 \begin{pmatrix}
1& 0 \\
0& -1\\
\end{pmatrix},\qquad
\pi_\la(x)=
 \begin{pmatrix}
0& \la \\
0& 0\\
\end{pmatrix}.
$$
It is easy to see that $\pi_{\la_1}\oplus \pi_{\la_2}$ is injective for every $\la_1\ne\la_2$. So we can assume that $H_{2,-1}$ is a subalgebra of $\rT_4$ and hence is of polynomial growth by Proposition~\ref{PGfd}. Then it is finitely $C^\infty$-generated by Corollary~\ref{finfCi}.
\end{exm}

\appendix

\section{Envelopes of free algebras}
\label{sec:efa}

Here we discuss  the structure of $\cF_{k}^{\,\mathsf{PG}}$ for general $k\in\N$. It is convenient to use the fact that the free algebra $\R\langle x_1,\ldots,x_k\rangle$ is isomorphic to the universal enveloping algebra $U(\ff_k)$, where~$\ff_k$ denotes a free real Lie algebra with $k$ generators. This section contains an explicit construction of $\wh U(\ff_k)^{\,\mathsf{PG}}$, which is topologically isomorphic $\cF_{k}^{\,\mathsf{PG}}$. First recall some notation from~\cite{ArNew}.

\subsection*{Notation and definitions from~\cite{ArNew}}

To describe $\wh U(\ff_k)^{\,\mathsf{PG}}$ we consider homomorphisms to Banach algebras of polynomial growth. Since the commutant of such an algebra is nilpotent \cite[Theorem 2.8 and Proposition 2.9]{ArOld}, we need a description of $U(\ff_k)$ that uses $[\ff_k,\ff_k]$. It is based on  polynomials in the operators $\ad e_i$, where $\{e_1,\ldots e_k\}$ is the set of generators of $\ff_k$. Denote by $\De$ the set of all pairs $(l,j)$ of positive integers such that $k\ge l>j\ge 1$.
 For $\de\in\De$ and $\be=(\be_1,\ldots, \be_l)\in\Z_+^l$ put
\begin{equation}\label{notg}
 g_{\de\be}\!:=(\ad{e_1})^{\be_1}\cdots (\ad{e_{l-1}})^{\be_{l-1}}(\ad{e_l})^{\be_l+1}(e_j),
\end{equation}
or, in an explicit form,
$$
g_{\de\be}=[e_1,[\ldots [e_2,[\cdots [e_l, e_j]\cdots]],
$$
where each of $e_i$ repeats $\be_i$ times except $e_l$ which repeats $\be_l+1$ times.

It is well known that the commutant $[\ff_k,\ff_k]$ is a free Lie algebra with the algebraic basis
\begin{equation}\label{basisg}
  \{g_{\de\be}\!:\,\de\in\De,\, \be\in\Z_+^l\};
\end{equation}
for a proof see, e.g., \cite[p.\,64, \S\,2.4.2, Corollary~2.16(ii)]{Ba21} or \cite[p.\,72, Lemma 2.11.16]{BK94}.
The PBW theorem implies that there is a linear isomorphism
\begin{equation}\label{Ufdec}
U(\ff_k)\cong \R[e_1,\ldots,e_k ]\otimes T(V),
\end{equation}
where~$V$ is the linear span of the set in~\eqref{basisg} and $T(V)$ denotes the tensor algebra of $V$.

Furthermore, each element of $V$ has the form
\begin{equation}\label{Ufdecee}
\sum_{\de\in\De}\Psi_\de(f_\de)([e_l,e_j]),
\end{equation}
where $f_\de\in \R[\mu_1,\ldots,\mu_l]$ and
\begin{equation}\label{orcalPsi}
\Psi_\de\!:\mu_1^{\be_1}\cdots  \mu_l^{\be_l} \mapsto \ad e_1^{\be_1}\cdots  \ad e_l^{\be_l}
\end{equation}
is a linear map from $\R[\mu_1,\ldots,\mu_l]$ to the space of linear endomorphisms of $[\ff_k,\ff_k]$. Therefore,
\begin{equation}\label{descV}
V\cong \bigoplus_{\de\in\De} \R[\mu_1,\ldots,\mu_l].
\end{equation}

Put $X\!:=\bigsqcup_{\de\in\De} X_\de$, where $X_{\de}\!:=\R^l$ and $\de=(l,j)$. We also put $X=\emptyset$ when $k=1$.
Let
\begin{equation}\label{cinX}
C^\infty(X)\!:=\bigoplus_{\de\in\De}  C^\infty(X_\de).
\end{equation}
The natural maps $\R[\mu_1,\ldots,\mu_l]\to C^\infty(X_\de)$, $\de\in\De$, induce an embedding $V\to C^\infty(X)$. Denote by $[T](C^\infty(X))$ the formal tensor algebra associated with the Fr\'echet space $C^\infty(X)$, that is,
$$
[T](C^\infty(X))=\prod_{m=0}^\infty C^\infty(X)^{\ptn m}
$$
with the concatenation as the multiplication; see the \cite[Definition 4.1]{ArNew}. It is also convenient to write
\begin{equation}\label{fotwal}
[T](C^\infty(X))=\prod_{m=0}^\infty C^\infty(X^m),
\end{equation}
where $X^m$ is the disjoint union of all possible $X_{\de_1}\times\cdots\times X_{\de_m}$.

The following definition is introduced in \cite[eq. (4.5)]{ArNew}:
\begin{equation}\label{defCifk}
C^\infty_{\ff_k}\!:=C^\infty(\R^k)\ptn [T](C^\infty(X))\qquad (k\in\N).
\end{equation}

The next two theorems were announced in \cite{ArNew} and proved there in the case when $k\le2$. A proof for arbitrary~$k$ was postponed due to many technical details. Our purpose here is to give a complete argument in the general case.

\begin{thm}\label{multCiffk}
\cite[Theorem 4.6]{ArNew}
Let $k\in \N$.

\emph{(A)}~The multiplication on $U(\ff_k)$ extends to a continuous multiplication on $C^\infty_{\ff_k}$.

\emph{(B)}~$C^\infty_{\ff_k}$ is  in $\mathsf{DTS}$ (see Definition~\ref{PSDTde}) and is therefore  a $\mathsf{PGL}$ algebra.
\end{thm}

The conclusion that $C^\infty_{\ff_k}$ belongs to $\mathsf{DTS}$ is omitted in the statement of \cite[Theorem 4.6]{ArNew}, but in fact the proof given there also provides  this result for $k\le 2$.

\begin{thm}\label{fuctcafree}
\cite[Theorem 4.7]{ArNew}
Let $k\in \N$. Then the algebra $C^\infty_{\ff_k}$ together with the embedding $U(\ff_k)\to C^\infty_{\ff_k}$ is the envelope with respect to $\mathsf{PG}$, i.e., $\wh U(\ff_k)^{\,\mathsf{PG}}\cong \cF_{k}^{\,\mathsf{PG}} \cong C^\infty_{\ff_k}$.
\end{thm}

Theorem~\ref{fuctcafree} easily follows from  Theorem~\ref{multCiffk} and Theorem~4.4
in~\cite{ArNew}, which asserts that a tuple  $(b_1, \ldots, b_k)$ of elements of a $\mathsf{PGL}$ algebra $B$ determines a linear continuous map $C^\infty_{\ff_k}\to B$ which extends
the homomorphism  $\te\!:U(\ff_k)\to B$ that maps $e_j$ to $b_j$ for each~$j$.

So the main purpose is to prove  Theorem~\ref{multCiffk}  for general $k$.
The central idea of the proof is to construct a homomorphism from $U(\ff_k)$ to
a $\mathsf{DTS}$  algebra admitting a linear extension to $C^\infty_{\ff_k}$ that is a topological isomorphism. In other words, we need a family of triangular representations of $\ff_k$ labelled by points of a union of manifolds and chosen so as to guarantee the topological injectivity. Because $\ff_k$  is free, there is a huge selection of representations but our goal is to choose a family that is not too large, if possible. To do this we first generalize some constructions in~\cite{ArNew}.

The argument for the case $k=2$, as presented in~\cite{ArNew}, makes significant use of the upper right entry of the representation matrix; cf. the proof of Theorem~\ref{multCifgqua}. For $k=2$, a single entry is sufficient since the set $\De$ consists of a single element. However, in general, we require an algebra with many different elements similar to the matrix unit in the upper right corner. In the next section, we consider a specific path algebra containing such elements.

\subsection*{A family of triangular representations well behaved on the commutant}

Let $m\in\N$ and $Q_m$ be the quiver (which is actually just a directed graph) whose vertices are labelled by $(\de,p)$ with $\de\in\De$ and $p=1,\ldots,m+1$ and
such that  there is a unique arrow with the source $(\de,p)$ and the target $(\de',p+1)$ for each
$\de,\de'\in\De$ and $p=1,\ldots,m$. Since $Q_m$ has finitely many vertices and arrows and no oriented cycles, the path algebra $\R Q_m$  is of polynomial growth by Proposition~\ref{patalnc} and we can use it for our purposes.

Denote the idempotent corresponding to $(\de,p)$ and the element corresponding to the arrow connecting $(\de,p)$ and $(\de',p+1)$ by $q_{\de p}$ and $v_{\de\de' p}$, respectively. Put also $v_{\de\de' m+2}=0$.
We need the standard relations, which follow from the definition of $\R Q_m$:
\begin{equation}\label{relRQm}
q_{\de p}v_{\de\de' p}=v_{\de\de' p}\qquad\text{and} \qquad v_{\de\de' p}q_{\de', p+1}=v_{\de\de' p}\qquad(p=1,\ldots,m).
\end{equation}
(Note that, unlike \cite{Be91}, we multiply the paths in direct order.) 

Put $Z_m\!:=\R^{|\De|(m+1)k}$ and denote by $v$ be the sum of all $v_{\de\de' p}$ and consider the family of Lie algebra homomorphisms $\{\te_{\mathbf{t}}\!:\ff_k\to \R Q_m;\,\mathbf{t}=(t_{\de pi})\in Z_m\}$
defined by
\begin{equation}\label{deft}
\te_{\mathbf{t}}(e_i)\!:=\sum_{\de,p} t_{\de pi}q_{\de p}+v \qquad (i=1,\ldots,k).
\end{equation}

The following two lemmas describe the action of $\te_{\mathbf{t}}$ on $V$ and $T(V)$.
Let $I$ be the two-sided ideal in $\R Q_m$ generated by all $v_{\de\de' p}$. Put also
\begin{equation}\label{sdedep}
s_{\de\de' pi}\!:=t_{\de pi}-t_{\de', p+1,i}
\end{equation}
 and
\begin{equation}\label{fde}
y_{\de_0\de\de' p}\!:=s_{\de\de' pl_0}- s_{\de\de' pj_0}\qquad(\de_0=(l_0,j_0)).
\end{equation}

\begin{lm}\label{teofg}
Let $\de_0=(l_0,j_0)\in\De$
and $\be=(\be_1,\ldots,\be_{l_0})\in\Z_+^{l_0}$. Then
\begin{equation}\label{gbete}
\te_{\mathbf{t}}(g_{\de_0\be})\in I\qquad\text{and}\qquad
\te_{\mathbf{t}}(g_{\de_0\be})-\sum_{\de,\de',p} y_{\de_0\de\de' p}\,s_{\de\de' p1}^{\be_1}\cdots s_{\de\de' pl_0}^{\be_{l_0}}\,v_{\de\de' p}\in I^2
\end{equation}
for every $\mathbf{t}$.
\end{lm}
\begin{proof}
We start by taking a look at the first commutators. Since all $q_{\de p}$ mutually commute, \eqref{relRQm} and\eqref{deft} imply that
\begin{equation}\label{elejte}
\te_{\mathbf{t}}([e_{l_0},e_{j_0}])=\sum y_{\de_0\de\de' p}\, v_{\de\de' p}.
\end{equation}
In particular, $\te_{\mathbf{t}}([e_{l_0},e_{j_0}])\in I$.

To obtain \eqref{gbete} for the higher commutators note that, for given $i$, $\de$, $\de'$ and $p$,
\begin{equation}\label{dedeprv}
[\te_{\mathbf{t}}(e_i),v_{\de\de' p}]-s_{\de\de' pi}v_{\de\de' p}\in I^2.
\end{equation}
Then for arbitrary tuple $(r_{\de\de' pi})$ of real numbers, we have
\begin{equation}\label{ldedeprv}
[\te_{\mathbf{t}}(e_i),\sum_{\de,\de',p} r_{\de\de' p}v_{\de\de' p}]-\sum_{\de,\de',p} s_{\de\de' pi}r_{\de\de' p}v_{\de\de'p}\in I^2.
\end{equation}
Since $\te_{\mathbf{t}}([e_{l_0},e_{j_0}])\in I$, it follows from~\eqref{notg} that $\te_{\mathbf{t}}(g_{\de_0\be})\in I$.
Using now~\eqref{ldedeprv} first with   $r_{\de\de' p}=y_{\de_0\de\de' p}$ and next making iterations, we conclude that the second relation in~\eqref{gbete} is  implied by \eqref{elejte}.
\end{proof}

Put also
$$
v_{\boldsymbol{\de}}\!:=v_{\de_1\de_2 1}\cdots v_{\de_{m}\de_{m+1} m} \qquad(\boldsymbol{\de}=(\de_1,\ldots,\de_m)\in \De^m).
$$

Now, knowing how  $\te_{\mathbf{t}}$ acts on the basis of $V$ we can easily find the action on the basis of $T(V)$, i.e., on the products of the iterated commutators in~\eqref{basisg}.

\begin{lm}\label{manyQkn}
Let $m\in\N$, $\mathbf{t}=(t_{\de pi})\in Z_m$,  $\boldsymbol{\de}=(\de_1,\ldots,\de_m)\in \De^m$ with $\de_p=(l_p,j_p)$ and $\be^p=(\be^p_1,\ldots,\be^p_{l_p})$ for each $p=1,\ldots,m$.
Then
\begin{equation}\label{wtelawn}
\te_{\mathbf{t}}(g_{\be^1\de_1}\cdots g_{\be^m\de_m})= \sum_{\boldsymbol{\de}'\in\De^m}\prod_{p=1}^m y_{\de_p\de'_{p}\de'_{p+1}p}\left(\prod_{i=1}^{l_p} s_{\de'_p\de'_{p+1}pi}^{\be_i^p}\right)\,v_{\boldsymbol{\de}'},
\end{equation}
where $\boldsymbol{\de}'=(\de'_1,\ldots,\de'_m)$.
\end{lm}
\begin{proof}
It is obvious that
$$
\te_{\mathbf{t}}(g_{\be^1\de_1}\cdots g_{\be^m\de_m}) =\te_{\mathbf{t}}(g_{\be^1\de_1})\cdots\te_{\mathbf{t}}(g_{\be^m\de_m}).
$$
Now we apply Lemma~\ref{teofg} for each $\de_p$.
Since $\te_{\mathbf{t}}(g_{\de_p\be})\in I$,  the product on the right-hand side is equal modulo $I^2$ to the product of $m$ sums of the form given in~\eqref{gbete}.

Multiplying the sums, we have that only coefficients of terms of the form
$$
v_{\de'_1\de'_2 1}\cdots v_{\de'_{m}\de'_{m+1} m}
$$
(which we denote by $v_{\boldsymbol{\de}'}$) do not vanish. Thus~\eqref{gbete} implies~\eqref{wtelawn}.
\end{proof}

\subsection*{A family with additional parameters}

Now we construct a more general family of representations and describe the action on the whole $U(\ff_k)$.

For $m\in\N$, $\boldsymbol{\la}=(\la_1,\ldots,\la_k)\in\R^k$ and $\mathbf{t}=(t_{\de pi})\in Z_m$ put \begin{equation}\label{defwte}
\wt\te_{\boldsymbol{\la},\mathbf{t}}\!:\ff_k\to \R Q_m\!: e_i\mapsto \la_i+\te_{\mathbf{t}}(e_i)\qquad (i=1,\ldots,k).
\end{equation}

To include the case of $m=0$ into consideration denote by $Q_0$ the $1$-vertex graph. It is obvious that  $\R Q_0\cong \R$. We put $\wt\te_{\boldsymbol{\la}}\!:e_i\mapsto \la_i$. (Note that that $\De$ is empty in this case.)

Write $U(\ff_k)$ in the form~\eqref{Ufdec}. We need an explicit form of this representation. Let $\boldsymbol{\mu}_{\boldsymbol{\de}}=(\mu_{pi})\in X_{\de_1}\times\cdots\times X_{\de_m}$, where $\boldsymbol{\de}=(\de_1,\ldots,\de_m)$, $p=1,\ldots m$  and  $i=1,\ldots l_p$. It follows from \eqref{Ufdecee} that each element of $U(\ff_k)$ can be written uniquely  as a finite sum
$$
\sum_{m=0}^\infty \sum_{\boldsymbol{\de}\in\De^m} \Phi_{\boldsymbol{\de}}(h_{\boldsymbol{\de}}),
$$
where
\begin{multline}\label{Phimdefk2n}
 \Phi_{\boldsymbol{\de}}\!: \R[\boldsymbol{\la},\boldsymbol{\mu}_{\boldsymbol{\de}}]\to U(\ff_k)\!:\\
 f_0\otimes f_{\de_1}\otimes  \cdots\otimes f_{\de_m} \mapsto \Psi_\emptyset(f_0)\Psi_{\de_1}(f_{\de_1})([e_{l_1},e_{j_1}])\cdots \Psi_{\de_m}(f_{\de_m})([e_{l_m},e_{j_m}]).
\end{multline}
Here $\Psi_\emptyset$ is the ordered calculus $\R[\boldsymbol{\la}]\to U(\ff_k)\!:\la_i\mapsto e_i$ and $\Psi_\de$ is defined in~\eqref{orcalPsi} (in particular, $\Phi_\emptyset=\Psi_\emptyset$).

Given $\boldsymbol{\de}=(\de_1,\ldots,\de_m)\in \De^m$ with $\de_p=(l_p,j_p)$ and $\boldsymbol{\de}'=(\de'_1,\ldots,\de'_m)\in\De^m$, put
\begin{equation}\label{sdel}
\mathbf{s}_{\boldsymbol{\de'\de}}\!:=(s_{\de'_p\de'_{p+1} pi})\qquad (p=1,\ldots,m,\,i=1,\ldots,l_p).
\end{equation}
(Note that not all available index values are included here.)

\begin{lm}\label{manyQwekn}
Let $\boldsymbol{\la}=(\la_1,\ldots,\la_k)\in\R^k$.

\emph{(A)}~Suppose that the hypotheses of Lemma~\ref{manyQkn} are satisfied and
$h\in\R[\boldsymbol{\la},\boldsymbol{\mu}_{\boldsymbol{\de}}]$. Then
\begin{equation}\label{PhimtrQmmann}
\wt\te_{\boldsymbol{\la},\mathbf{t}}(\Phi_{\boldsymbol{\de}}(h))=\sum_{\boldsymbol{\de}'\in\De^m}\left(\prod_{p=1}^m y_{\de_p\de'_{p}\de'_{p+1}p}\right) h(\la_1+t_{\de'_1 11},\ldots,\la_k+t_{\de'_1 1k},\,\mathbf{s}_{\boldsymbol{\de'\de}}) v_{\boldsymbol{\de}'}.
\end{equation}

\emph{(B)}~If $m=0$ and  $h\in\R[\boldsymbol{\la}]$, then
$\wt\te_{\boldsymbol{\la}}(\Phi_\emptyset(h))= h(\la_1,\ldots,\la_k)$.
\end{lm}

\begin{proof}
(A)~Suppose first that  $h$ is a power product in $\R[\boldsymbol{\la},\boldsymbol{\mu}_{\boldsymbol{\de}}]$, i.e.,
$$
h=\la_1^{\al_1}\cdots\la_k^{\al_k} \prod_{p=1}^{m} \prod_{i=1}^{l_p}\mu_{pi}^{\be^p_i}.
$$
for some $\al_1,\ldots,\al_k$ and  $\be^p=(\be^p_1,\ldots,\be^p_{l_p})$.
By the definition of $\Phi_{\boldsymbol{\de}}$, we have  that
$$
\Phi_{\boldsymbol{\de}}(h)= e_1^{\al_1}\cdots e_k^{\al_k} g_{\be^1\de_1}\cdots g_{\be^m\de_m}.
$$

Note that $q_{\de p}v_{\de_1'\de_2' 1}=v_{\de_1'\de_2' 1}$ when $\de=\de_1'$ and $p=1$ and equals $0$ otherwise; see~\eqref{relRQm}. So, in view of $vv_{\boldsymbol{\de}'}=0$, we have from \eqref{deft} and~\eqref{defwte} that for every $\boldsymbol{\de}'\in\De^m$
$$
\wt\te_{\boldsymbol{\la},\mathbf{t}}(e_i)v_{\boldsymbol{\de}'}=(\la_i+t_{\de'_1 1i})v_{\boldsymbol{\de}'}.
$$
Therefore,
$$
\wt\te_{\boldsymbol{\la},\mathbf{t}}(e_1^{\al_1}\cdots e_k^{\al_k})v_{\boldsymbol{\de}'}=(\la_1+t_{\de'_1 11})^{\al_1}\cdots(\la_k+t_{\de'_1 1k})^{\al_k}v_{\boldsymbol{\de}'}.
$$
Note also that $\wt\te_{\boldsymbol{\la},\mathbf{t}}(a)=\te_{\mathbf{t}}(a)$ for every $a\in [\ff_k,\ff_k]$. Then by Lemma~\ref{manyQkn},
$$
\wt\te_{\boldsymbol{\la},\mathbf{t}}(\Phi_{\boldsymbol{\de}}(h))=\prod_{i'=1}^{k} (\la_{i'}+t_{\de'_1 1{i'}})^{\al_{i'}}  \left(\sum_{\boldsymbol{\de}'\in\De^m}\prod_{p=1}^m y_{\de_p\de'_{p}\de'_{p+1}p}\left(\prod_{i=1}^{l_p} s_{\de'_p\de'_{p+1}pi}^{\be_i^p}\right)\right)v_{\boldsymbol{\de}'}.
$$
Thus \eqref{PhimtrQmmann} holds for monomials. By the linearity of both sides of the equality, it also  holds for every $h\in\R[\boldsymbol{\la},\boldsymbol{\mu}_{\boldsymbol{\de}}]$.

The argument for Part~(B) is straightforward and so is omitted.
\end{proof}

We also need the simple case, where the number of factors is greater than~$m$.

\begin{lm}\label{annhn}
Let $\breve{m}\in\N$, $\breve{\boldsymbol{\de}}\in\De^{\breve{m}}$, $\boldsymbol{\la}\in\R^k$ and $h\in\R[\boldsymbol{\la},\boldsymbol{\mu}_{\breve{\boldsymbol{\de}}}]$.
Suppose that $m<\breve{m}$.
Then $\wt\te_{\boldsymbol{\la},\mathbf{t}}(\Phi_{\breve{\boldsymbol{\de}}}(h))=0$ for all  $\mathbf{t}\in Z_m$. If $m=0$,  then $\wt\te_{\boldsymbol{\la}}(\Phi_{\breve{\boldsymbol{\de}}}(h))=0$.
\end{lm}

\begin{proof}
It follows from~\eqref{Phimdefk2n} that $\Phi_{\breve{\boldsymbol{\de}}}(h)$ is a sum of products of $\breve{m}+1$ factors, where all factors except first are of the form $g_{\be\de}$. It follows from~\eqref{gbete} that after applying of $\wt\te_{\boldsymbol{\la},\mathbf{t}}$ each of the summands contains a product of the form $v_{\de_1\de_2 1}\cdots v_{\de_{\breve{m}}\de_{\breve{m}+1} \breve{m}}$, which vanishes since $\breve{m}>m$.
\end{proof}

\subsection*{Proof of the structural theorem}

\begin{proof}[Proof of Theorem~\ref{multCiffk}]
We prove Parts (1) and (2) simultaneously. Since $\mathsf{DTS}$ is stable under passing to closed subalgebras, it suffices to construct a homomorphism from $U(\ff_k)$  to a $\mathsf{DTS}$  algebra admitting a linear extension to $C^\infty_{\ff_k}$ that is a topological isomorphism.

For given $m\in \N$ put
$$
B_m\!:=C^\infty(\R^k\times Z_m,\R Q_m),
$$
where $Z_m\!:=\R^{|\De|(m+1)k}$. We also add the case of $m=0$, where $Z_0$ is a single-point set and $\R Q_0=\R$ by definition.
Consider the family
$$
\{\wt\te_{\boldsymbol{\la},\mathbf{t}}\!:U(\ff_k)\to \R Q_m\!:\,\boldsymbol{\la}\in\R^k,\,\mathbf{t}\in Z_m\}
$$
of homomorphisms defined in~\eqref{defwte}. Similarly, for $m=0$ consider  the family
$$
\{\wt\te_{\boldsymbol{\la}}\!:U(\ff_k)\to C^\infty(\R^k)\!:\,\boldsymbol{\la}\in\R^k\}, 
$$
where $e_i\mapsto \la_i$ as above.
Varying $\boldsymbol{\la}$ and $\mathbf{t}$ we have the homomorphism
$$
\pi_m\!:U(\ff_k)\to B_m\!:\,
\pi_{m}(a)(\boldsymbol{\la},\mathbf{t})\!:= \wt\te_{\boldsymbol{\la},\mathbf{t}}(a)\qquad (a\in U(\ff_k))
$$
for $m\ge 1$ and similarly for $m=0$.

Consider also the homomorphism
$$
\rho\!:U(\ff_k)\to \prod_{m=0}^{\infty} B_m \!:a\mapsto (\pi_m(a)).
$$
Since $\prod B_m$ is a $\mathsf{PGL}$ algebra, it follows from \cite[Theorem~4.4
]{ArNew} that $\rho$ extends to a continuous linear map from $C^\infty_{\ff_k}$ to $\prod B_m$.
By Definition~\ref{PSDTde}, every closed subalgebra of $\prod B_m$  belongs to $\mathsf{DTS}$.
Thus to complete the proof we need to show that the extension of $\rho$ is topologically injective.

By the definition (see \eqref{defCifk} and~\eqref{cinX}), $C^\infty_{\ff_k}=C^\infty(\R^k)\ptn [T](C^\infty(X))$ as a locally convex space, where $X=\coprod X_\de$.
It follows from the explicit description of the formal tensor algebra  that  $C^\infty_{\ff_k}=\prod_{m=0}^\infty C_m$, where $C_m\!:=C^\infty(\R^k\times X^m)$; see~\eqref{fotwal}. Then $\rho$ can be written as $\prod C_m \to \prod B_m$. It follows from Lemma~\ref{annhn} that $\rho$ is determined by a lower triangular infinite matrix and so, by \cite[Lemma 4.12]{ArNew}, it is sufficient to check that $C_m\to B_m$ is topologically injective for every $m$.

Note first that $C_0\to B_0$ is a topological isomorphism.
Assume next that $m\ge 1$. We treat $B_m$ as a free $C^\infty(\R^k\times Z_m)$-module and denote by $Y_m$ the submodule $\sum_{\boldsymbol{\de}' \in \De^m}B_m v_{\boldsymbol{\de}'}$. It is obvious that $Y_m$ is a direct factor in $B_m$. So  it suffices to show that the composition of $\pi_m$ and the projection of the whole $C^\infty_{\ff_k}$ onto $Y_m$ is topologically injective; see, e.g., Part~(B) of \cite[Lemma 3.10
]{ArNew}. We denote this composition by $\pi'_m$.

Lemma~\ref{manyQwekn} asserts that~\eqref{PhimtrQmmann} holds for every  $\boldsymbol{\de}\in \De^m$ and every $h_{\boldsymbol{\de}}\in\R[\boldsymbol{\la},\boldsymbol{\mu}_{\boldsymbol{\de}}]$. The existence of an extension to an ordered $C^\infty$-calculus given by \cite[Theorem 4.4]{ArNew} implies that \eqref{PhimtrQmmann}  holds for every $h_{\boldsymbol{\de}}\in C^\infty(\R^k\times X_{\de_1}\times\cdots\times X_{\de_m})$. Note that the cardinality of $\De^m$ equals $m|\De|$ and identify  $Y_m$ with $C^\infty(\R^k\times Z_m)^{m|\De|}$. 
Identifying also $X^m$ with $\coprod_{\boldsymbol{\de}\in\De^m} (X_{\de_1}\times\cdots\times X_{\de_m})$ and $C_m$ with $\prod_{\boldsymbol{\de}\in\De^m} C^\infty(\R^k\times X_{\de_1}\times\cdots\times X_{\de_m})$, we conclude that $\pi'_m$ has the form
\begin{equation}\label{dedepr}
(h_{\boldsymbol{\de}})\mapsto\left((\boldsymbol{\la},\mathbf{t})\mapsto \sum_{\boldsymbol{\de}} \left(\prod_{p=1}^m y_{\de_p\de'_{p}\de'_{p+1}p}\right) h_{\boldsymbol{\de}}(\la_1+t_{\de'_1 11},\ldots,\la_k+t_{\de'_1 1k},\,\mathbf{s}_{\boldsymbol{\de'\de}})\right)_{\boldsymbol{\de}'},
\end{equation}
where
$$
y_{\de_p\de'_{p}\de'_{p+1}p}=t_{\de'_p, p, l_{\de_p}}-t_{\de'_{p+1}, p+1, l_{\de_p}}-t_{\de'_p, p, j_{\de_p}}+t_{\de'_{p+1}, p+1, j_{\de_p}} \qquad (\de_p=(l_p,j_p))
$$
by~\eqref{sdedep} and~\eqref{fde}, and 
$$
\mathbf{s}_{\boldsymbol{\de'\de}}=(t_{\de pi}-t_{\de', p+1,i})\qquad (p=1,\ldots,m,\,i=1,\ldots,l_p)
$$
by~\eqref{sdel}.
The map in~\eqref{dedepr} is the composition of the map $C_m\to Y_m$ induced by a finite-dimensional linear operator and the multiplication by the matrix $F\!:=\left(\prod_{p=1}^m y_{\de_p\de'_{p}\de'_{p+1}p}\right)_{\boldsymbol{\de}'\boldsymbol{\de}}$. (Here we fix a linear order on the set~$\De^m$.)

The linear operator that induces the first map has the form
$$
((\la_{i'}),(t_{\de pi}))\mapsto ((\la_{i'}+t_{\de'_1 1{i'}}),(t_{\de'_p pi}-t_{\de'_{p+1}, p+1,i})),
$$
where $i'=1,\ldots, k$, $p=1,\ldots,m$ and  $i=1,\ldots, l_p$. It is easy to see that this operator is surjective. Therefore the first map is topologically injective.

The multiplication by $F$ is injective since every  tuple of functions that it takes to $0$ obviously vanishes at the points, where $\det F\ne 0$, and  also vanishes at the remaining points by continuity (since  $\det F$ is a polynomial). Also,  the range of this map is a submodule of a finitely generated free $C^\infty(\R^k\times Z_m)$-module and, moreover, this submodule generated by a finite tuple of  polynomials. It is well known that each such submodule is closed; see, e.g.,  \cite[Chapitre\,VI, p.\,119, Corollaire~1.5]{Tou72}. Thus, we have a topological isomorphism onto the range by the inverse mapping theorem for Fr\'echet spaces. Being a composition of two topologically injective maps, $\pi'_m\!:C_m\to Y_m$ is topologically injective and so is $C_m\to B_m$.

Thus $C_m\to B_m$ is topologically injective for every $m$. As noted above, this implies that $\rho\!:C^\infty_{\ff_k}\to\prod B_m$ is also topologically injective and this completes the proof.
\end{proof}

\begin{rem}
In \cite{ArNew} the author proved Theorem~\ref{multCiffk} for $k\le2$. Look at the case $k=2$.
Compared to the proof in the general case given here, \cite{ArNew} contains the following simplifications. First, it is easy to see that $X=\R^2$ and $\R Q_m\cong \rT_{m+1}$. The second simplification is a technical one. Note that both $\te_{\mathbf{t}}(e_1)$ and $\te_{\mathbf{t}}(e_2)$ defined in~\eqref{deft} can be written as sums of a diagonal matrix and a Jordan block with zero eigenvalue. But in \cite[eq.~(4.9)]{ArNew}  a shorter formula was used: the image of $e_1$ is assumed to be diagonal. The rest of the argument is very similar to that contained here.
\end{rem}

\begin{rem}
Recall that the \emph{Whitney ideal} $W_{Y/\R^m}$ corresponding to a closed subset $Y$ of $\R^m$ is the set of all functions in $C^\infty(\R^m)$ that are flat on~$Y$, i.e., for each function in $W_{Y/\R^m}$ there is an open neighbourhood of~$Y$ such that all partial derivatives vanish at every point. Elements of the $C^\infty$-differentiable algebra $C^\infty(\R^m)/W_{Y/\R^m}$ can be treated as `functions' on a $C^\infty$-differentiable space, which is denoted by $\mathbf{W}_{Y/\R^m}$; see, e.g., \cite[p.\,60, Corollary~5.10]{NaSa}. We denote this algebra by $C^\infty(\mathbf{W}_{Y/\R^m})$.

In at least four cases we can naturally identify the underlying locally convex space of a  finitely $C^\infty$-generated algebra with $C^\infty(\mathbf{W}_{\Om/\R^m})$ for some $m$ and $\Om$: the quantum plane, the quantum group $\SL_q(2,\R)$ and a triangular finite-dimensional Lie algebra (see Theorems~4.3,  4.10 and~4.12, in \cite{ArNew}) and also a deformation of the quantum group $\fa\ff_1$;  see Theorem~\ref{multCifgqua}. On the other hand, in the free algebra case the underlying space is more complicated (Theorem~\ref{fuctcafree}). Thus it is interesting to know under which conditions  the underlying space has the form $C^\infty(\mathbf{W}_{\Om/\R^m})$. For example, is the finiteness of the Gelfand--Kirillov dimension is necessary?
\end{rem}

\section*{Conflict of interest}

The author states that there is no conflict of interest.

\section*{Data Availability}

My manuscript has no associated data.

\end{document}